\documentclass[a4paper, 10pt, final]{amsart}

\usepackage[latin1]{inputenc}
\usepackage[T1]{fontenc}
\usepackage[english]{babel}
\usepackage{amsmath}

\usepackage{lmodern}

\usepackage[OT2,T1]{fontenc}

\usepackage{amsmath, amsthm, amsfonts}

\DeclareMathOperator{\mult}{mult}
\DeclareMathOperator{\Lie}{Lie}
\DeclareMathOperator{\weight}{weight}
\DeclareMathOperator{\id}{id}
\DeclareMathOperator{\KZ}{KZ}

\DeclareMathOperator{\red}{red}

\DeclareMathOperator{\Ad}{Ad}
\DeclareMathOperator{\depth}{depth}
\DeclareMathOperator{\M}{M}

\DeclareMathOperator{\dR}{dR}
\DeclareMathOperator{\Spec}{Spec}
\DeclareMathOperator{\B}{B}
\DeclareMathOperator{\gr}{gr}

\DeclareMathOperator{\inv}{inv}

\DeclareMathOperator{\mot}{mot}
\DeclareMathOperator{\De}{De}

\DeclareMathOperator{\dec}{dec}

\DeclareMathOperator{\un}{un}
\DeclareMathOperator{\Aut}{Aut}

\DeclareMathOperator{\Gal}{Gal}

\DeclareMathOperator{\GRT}{GRT}

\DeclareMathOperator{\har}{har}

\theoremstyle{definition}

\newtheorem{Théorème}{Théorème}[section]
\newtheorem{Proposition}[Théorème]{Proposition}

\newtheorem{Notation}[Théorème]{Notation}

\newtheorem{Nota Bene}[Théorème]{Nota Bene}

\newtheorem{Remark}[Théorème]{Remark}

\newtheorem{Fact}[Théorème]{Fact}

\newtheorem{Proposition-Definition}[Théorème]{Proposition-Definition}
\newtheorem{Example}[Théorème]{Example}

\newtheorem{Lemma}[Théorème]{Lemma}

\newtheorem{Corollary}[Théorème]{Corollary}

\newtheorem{N.B.}[Théorème]{N.B.}

\input cyracc.def
\DeclareFontFamily{U}{russian}{}
\DeclareFontShape{U}{russian}{m}{n}
        { <5><6> wncyr5
        <7><8><9> wncyr7
        <10><10.95><12><14.4><17.28><20.74><24.88> wncyr10 }{}
\DeclareSymbolFont{Russian}{U}{russian}{m}{n}
\DeclareSymbolFontAlphabet{\mathcyr}{Russian}
\makeatletter
\let\@math@cyr\mathcyr
\renewcommand{\mathcyr}[1]{\@math@cyr{\cyracc #1}}
\makeatother
\newcommand{\sh}{\mathcyr{sh}} 

\makeatletter
\newcounter {subsubsubsection}[subsubsection]
\renewcommand\thesubsubsubsection{\thesubsubsection .\@alph\c@subsubsubsection}
\newcommand\subsubsubsection{\@startsection{subsubsubsection}{4}{\z@}%
                                     {-3.25ex\@plus -1ex \@minus -.2ex}%
                                     {1.5ex \@plus .2ex}%
                                     {\normalfont\normalsize\bfseries}}
\newcommand*\l@subsubsubsection{\@dottedtocline{3}{10.0em}{4.1em}}
\newcommand*{\subsubsubsectionmark}[1]{}
\makeatother

\numberwithin{equation}{section}

\usepackage[top=2.5cm, bottom=2.5cm, left=2.5cm, right=2.5cm]{geometry}

\author{David Jarossay}

\address{Mathematics Department, Ben Gurion University of the Negev, Be'er-Sheva`, Israel}

\email{jarossay@post.bgu.ac.il}

\begin{document}

\title{Depth reductions for associators}

\maketitle

\begin{abstract}
We prove that for any associator, two specific families of coefficients of the associator can be expressed in terms of coefficients of lower depth. Combining these results to our notions of adjoint $p$-adic multiple zeta values and multiple harmonic values, we obtain a new point of view on the question of relating $p$-adic and finite multiple zeta values, and a few other application to the study of $p$-adic multiple zeta values via explicit formulas.
\end{abstract}

\setcounter{section}{-1}

\tableofcontents

\section{Introduction}

\subsection{Real and $p$-adic multiple zeta values}

Multiple zeta values (MZV's) are the following real numbers : for any positive integers $d$ and $n_{i}$ ($1 \leq i \leq d$) such that $n_{d}\geq 2$,

\begin{equation} \label{eq:mzv} \zeta(n_{1},\ldots,n_{d})
= \sum_{0<m_{1}<\cdots<m_{d}} \frac{1}{m_{1}^{n_{1}} \cdots m_{d}^{n_{d}}} .
\end{equation}

We say that $(n_{1},\ldots,n_{d})$ has weight $n=n_{1}+\cdots+n_{d}$ and depth $d$. Denoting by $(\epsilon_{n},\ldots,\epsilon_{1}) = (\underbrace{0, \ldots, 0}_{n_{d}-1}, 1,\ldots,\underbrace{0, \ldots, 0}_{n_{1}-1},1)$, we have

\begin{equation} \label{eq:mzv integral}
\zeta(n_{1},\ldots,n_{d}) = (-1)^{d} \int_{t_{n}=0}^{1} \frac{dt_{n}}{t_{n}-\epsilon_{n}} \cdots \int_{t_{1}=0}^{t_{2}} \frac{dt_{1}}{t_{1}-\epsilon_{1}} .
\end{equation}

$p$-adic multiple zeta values ($p$MZV's) are numbers $\zeta_{p,\alpha}(n_{1},\ldots,n_{d}) \in \mathbb{Q}_{p}$ (where $\alpha$ is a parameter in $(\mathbb{Z} \cup \{\pm\infty\}) - \{0\}$, according to our generalized definition) defined as certain $p$-adic analogues of the integrals (\ref{eq:mzv}) \cite{Furusho MZV1, Furusho MZV2, Deligne Goncharov, Unver MZV, J1, J3}.

\subsection{Multiple zeta values and associators}

An associator is a non-commutative formal power series satisfying certain polynomial equations \cite{Drinfeld}. Multiple zeta values and their $p$-adic analogue provide an example of associator.

For any ring $R$, let $R \langle\langle e_{0},e_{1} \rangle\rangle$ be the $R$-algebra of formal power series on the non-commuting variables $e_{0},e_{1}$ with coefficients in $R$. We write elements $f$ of $R\langle\langle e_{0},e_{1} \rangle\rangle$ as follows : $f=\sum\limits_{w\text{ word on }\{e_{0},e_{1}\}} f[w]w$ with  $f[w] \in R$, i.e. we denote by $f[w] \in R$ the coefficient of a word $w$ in $f$.

Let $K$ be a field of characteristic $0$, and $\mu \in K$. One has an affine scheme $M_{\mu}$ of associators with parameter $\mu$ over $K$. One also has an algebraic group $\GRT_{1}$, such that $M_{\mu}$ is a torsor under that group, which is isomorphic to $M_{0}$ as an affine scheme. The points of $M_{0}$ are called degenerated associators and the points of $M_{\mu}$ for $\mu \not= 0$ are called non-degenerated associators. For any $K$-algebra $R$, we have inclusions $M_{\mu}(R) \subset R \langle\langle e_{0},e_{1} \rangle\rangle$ and $\GRT_{1}(R) \subset R \langle\langle e_{0},e_{1} \rangle\rangle$, which are functorial and compatible to the torsor structure.

Let $\Phi_{\KZ} \in M_{2i\pi}(\mathbb{R})$ be the  Knizhnik-Zamolodchikov (or KZ) associator defined in \cite{Drinfeld}, \S2. We have for all integers $d$ and $n_{i}$ $(1 \leq i \leq d)$ we have $\zeta(n_{1},\ldots,n_{d}) = (-1)^{d} \Phi_{\KZ}[e_{0}^{n_{d}-1}e_{1}\cdots e_{0}^{n_{1}-1}e_{1}]$.

The fact that $\Phi_{\KZ}$ is an associator means the following. We have a connexion, called the Knizhnik-Zamolodchikov connection, on a trivial bundle on $\mathcal{M}_{0,4} = \mathbb{P}^{1} - \{0,1,\infty\}$ and $\mathcal{M}_{0,5} = (\mathbb{P}^{1} - \{0,1,\infty\})^{2} - \Delta$. The integration of this connection along certain special paths gives a trivial result because those paths are contractile and has at the same time a non-trivial expression (\cite{Drinfeld}, \S2).

For each $\alpha$, we have a $p$-adic analogue $\Phi_{p,\alpha} \in \mathbb{Q}_{p}\langle\langle e_{0},e_{1} \rangle\rangle$ of $\Phi_{\KZ}$, such that
$\zeta_{p,\alpha}(n_{1},\ldots,n_{d}) = (-1)^{d} \Phi_{p,\alpha}[e_{0}^{n_{d}-1}e_{1}\cdots e_{0}^{n_{1}-1}e_{1}]$ for all integers $d$ and $n_{i}$ $(1 \leq i \leq d)$. We have $\Phi_{p,\alpha}\in \GRT_{1}(\mathbb{Q}_{p})$ \cite{Unver}.

The fact that $\Phi_{\KZ}$ is in $M_{2\pi i}(\mathbb{R})$, resp. $\Phi_{p,\alpha}$ is in $\GRT_{1}(\mathbb{Q}_{p})$ amounts to polynomial equations over $\mathbb{Q}$ satisfied by MZV's resp. $p$MZV's. It is conjectured that they generate the ideal of all polynomial equations over $\mathbb{Q}$ satisfied by MZV's resp. $p$MZV's, and that the algebraic equations over $\mathbb{Q}$ among MZV's modulo $\zeta(2)$ are the same with the algebraic equations over $\mathbb{Q}$ of $p$-adic MZV's (see for example \cite{Andre} \S25.4).

\subsection{Multiple zeta values as periods, and the depth filtration}

Equation (\ref{eq:mzv}) shows that MZV's are \emph{periods} in the sense of algebraic geometry. More precisely, it shows that they are Betti - de Rham periods of the pro-unipotent fundamental groupoid ($\pi_{1}^{\text{un}}$) of $\mathcal{M}_{0,4}=\mathbb{P}^{1} - \{0,1,\infty\}$ (\cite{Deligne Goncharov} \S5.16). Similarly, $p$-adic multiple zeta values are reductions of $p$-adic periods of $\pi_{1}^{\text{un}}(\mathbb{P}^{1} - \{0,1,\infty\})$ \cite{Yamashita}.

This means that $\pi_{1}^{\text{un}}(\mathbb{P}^{1} - \{0,1,\infty\})$ can be taken as a tool from algebraic geometry to study the $\mathbb{Q}$-algebra generated by MZV's, resp. $p$MZV's. More precisely, it is conjectured that the $\mathbb{Q}$-algebra generated by MZV's resp. $p$MZV's is equal to its lift defined by the motivic $\pi_{1}^{\text{un}}(\mathbb{P}^{1} - \{0,1,\infty\})$ (\cite{Deligne Goncharov}, \S5).

Multiple zeta values are an important example of periods, because the corresponding motives are mixed Tate motives, because of their relation with the theory of associators, their appearance in quantum field theory.

Among the study of the $\mathbb{Q}$-algebra generated by MZV's, resp. $p$MZV's, one theme is the role played by the depth filtration. 

We say that an associator $\Phi \in M_{\mu}(K)$ admits a depth reduction at a certain linear combination $\sum_{i} a_{i}w_{i}$ of words of a given depth $d$ if $\Phi[a_{i}w_{i}]$ is a $\mathbb{Q}[\mu]$-polynomial of coefficients of $\Phi$ at words of depth strictly less than d. In the case of $\Phi_{\KZ}$ (where $\mu=2\pi i$), the simplest example is the following famous formula due to Euler :
\begin{equation} \label{eq:Euler} \forall n\geq 1,\text{ } \zeta(2n) = \frac{(-1)^{n+1}B_{2n}(2\pi)^{2n}}{2(2n)!}.
\end{equation}
The $p$-adic analogue of (\ref{eq:Euler}) is that $ \forall n \geq 1$, $\zeta_{p,\alpha}(2n)=0$. 

Studying depth reductions amounts to study the \emph{depth-graded multiple zeta values}. They appear in quantum field theory and in relation with double shuffle relations \cite{IKZ}, and they play a role in the study of the pro-unipotent fundamental group of moduli spaces of curves of genus one. 

To our knowledge, associator equations are usually not used to deal with depth reductions because they are not directly adapted to the depth filtration, unlike double shuffle relations. In this paper, we are going to reformulate certain associator equations in order to make them naturally adapted to the depth filtration. We will regard these equations as depending of two variables : an associator $\Phi$ and a variant $\Phi_{\infty}$ of $\Phi$. Moreover, we will consider intrinsically $\Phi^{-1}e_{1}\Phi$ and $\Phi^{-1}_{\infty}e_{\infty}\Phi_{\infty}$, where $e_{\infty}$ is such that $e_{0} + e_{1} + e_{\infty}=0$.

\subsection{The relation between $p$-adic and finite multiple zeta values}

Let $\mathcal{P}$ be the set of prime numbers. The ring of integers modulo infinitely large primes \cite{Ko} is the following ring :

$$ \mathcal{A} = \Big( \prod_{p \in \mathcal{P}} \mathbb{Z}/p\mathbb{Z} \Big) \Big/ \Big( \underset{p \in \mathcal{P}}{\bigoplus} \mathbb{Z}/p\mathbb{Z} \Big) . $$

Finite multiple zeta values (fMZV's) are defined by Kaneko and Zagier, after the work of other authors in particular Hoffman and Zhao, as the following elements of $\mathcal{A}$ : for any positive integers $d$ and $n_{i}$ ($1 \leq i \leq d)$,

$$ \zeta_{\mathcal{A}}(n_{1},\ldots,n_{d}) = \bigg(  \sum_{0<m_{1}<\cdots<m_{d}<p} \frac{1}{m_{1}^{n_{1}}\cdots m_{d}^{n_{d}}} \mod p \bigg)_{p \in \mathcal{P}} $$

There are striking analogies between finite multiple zeta values and real modulo $\zeta(2)$ resp. $p$-adic multiple zeta values. Kaneko and Zagier conjecture that there is an isomorphism between the $\mathbb{Q}$-algebra generated by finite MZV's and the $\mathbb{Q}$-algebra generated by MZV's moded out by the ideal $(\zeta(2))$, and a non-obvious explicit formula for that isomorphism (see \S5.2).

Relating $p$-adic and finite multiple zeta values was one of our main motivations for our work on the explicit computation of $p$-adic multiple zeta values \cite{J1, J2, J3} and the relation between this explicit computation and the algebraic theory \cite{J4, J5, J6}.

In \cite{J2}, we have computed $p$-adic multiple zeta values and found a $p$-adic analogue of the formula (\ref{eq:mzv}). This means that we have found a formula which is expressed in terms of the multiple harmonic sums, the numbers 
$\displaystyle \sum\limits_{0<m_{1}<\ldots<m_{d}<m} \frac{1}{m_{1}^{n_{1}} \cdots m_{d}^{n_{d}}}$. 
Using our computation and an integrality result of $p$-adic multiple zeta values, Akagi, Hirose and Yasuda have shown that the reduction of $p$-adic multiple zeta values can be expressed in terms of finite multiple zeta values \cite{AHY}.

We will point out in this paper a new interpretation of the question of relating $p$-adic and finite multiple zeta values, which arises naturally from our study of $p$-adic multiple zeta values. We will explain that this question is a simple quotient of a more general and deeper question, related to a theme studied in \cite{J6}. The study of depth reductions for associators gives a way to tackle this deeper question and also has other applications to our study of $p$-adic multiple zeta values.

\subsection{Main result}

In this paper we reformulate part of the associator equations in a way and which is adapted to the depth filtration in a certain sense. As a consequence, we prove two families of depth reductions :

\textbf{Theorem.} \emph{Let $\Phi$ be an associator and let any positive integers $d$ and $n_{i}$ $(1 \leq i \leq d)$. We have :}

\emph{(i) $\Phi[e_{0}^{n_{d}-1}e_{1}\cdots e_{0}^{n_{1}-1}e_{1}] + (-1)^{n_{d}+\cdots+n_{1}}\Phi[e_{0}^{n_{1}-1}e_{1}\cdots e_{0}^{n_{d}-1}e_{1}]$ admits a depth reduction.}

\emph{(ii) If $n_{1}+\cdots+n_{d}-d$ is odd, then $\Phi[e_{0}^{n_{1}-1}e_{1}\cdots e_{0}^{n_{d}-1}e_{1}]$ admits a depth reduction.}

This theorem can also be obtained by combining the fact that these depth reductions hold for solutions to the double shuffle equations (\cite{Yasuda} for (i), \cite{IKZ} for (ii)), and the fact that associator satisfy the double shuffle relations \cite{Furusho assoc ds}. The present proof is simpler, it provides an information on the $p$-adic norms of the rational coefficients, and it has applications to our theory of $p$-adic multiple zeta values via explicit formulas. These applications are explained in \S5. The $d=1$ case of (ii) is Euler's formula (\ref{eq:Euler}). We call (i) the \emph{adjoint depth reduction} and (ii) is the \emph{parity depth reduction}.

We review definitions in \S1, we write preliminary 
computations in \S2, we treat the case of $\GRT_{1}=M_{0}$ in \S3 and the case of $M_{\mu}$ with $\mu\not= 0$ in \S4. We explain the applications to $p$MZV's in \S5. In \S5.1, we apply the techniques used in the previous proofs to clarify the relation between $p$-adic multiple zeta values and adjoint $p$-adic multiple zeta values, which are defined in \cite{J4}, implicitly present in \cite{J1, J2, J3}. In \S5.2 we give a new point of view on the relation between $p$-adic and finite multiple zeta values. In \S5.3, by adapting some of our proofs we deduce a property related to the problem of the interpolation of $p$-adic multiple zeta values.

\emph{Acknowledgments.} I thank the referee for his useful corrections and suggestions. I thank Benjamin Enriquez and Erik Panzer for discussions. This work has been done at Universit\'{e} Paris Diderot with support of ERC grant n°257638, at Universit\'{e} de Strasbourg with support of Labex IRMIA. It has been revised at Universit\'{e} de Gen\`{e}ve with support of NCCR SwissMAP and at Ben Gurion University of the Negev with support of the ISF grant n°87590031 of Ishai Dan-Cohen.

\section{Definitions and notations}

We review the framework of the de Rham pro-unipotent fundamental groupoid of $\mathcal{M}_{0,4} \simeq \mathbb{P}^{1} - \{0,1,\infty\}$ and $\mathcal{M}_{0,5}$, and the notion of associators.
 
\subsection{The de Rham pro-unipotent fundamental groupoid of $\mathcal{M}_{0,4}$ and $\mathcal{M}_{0,5}$}

\subsubsection{Generalities}

The notion of pro-unipotent fundamental group of a smooth algebraic variety $X$ has been introduced in \cite{Deligne}, with a motivation from Grothendieck's \emph{Esquisse d'un programme}, which led to consider the cases of $X= \mathcal{M}_{g,n}$ the moduli space of curves of genus $g$ with $n$ marked points ; more particularly, the case $g \in \{0,1\}$, and, even more particularly, the projective line minus three points $\mathcal{M}_{0,4} \simeq \mathbb{P}^{1} - \{0,1,\infty\}$, $\mathcal{M}_{0,5} \simeq \{(y_{1},y_{2}) \in (\mathbb{P}^{1} - \{0,1,\infty\})^{2} \text{ }|\text{ }y_{1}\not=y_{2}\}$, $\mathcal{M}_{1,1}$ and $\mathcal{M}_{1,2}$. Later, the $\pi_{1}^{\un}$ of those varieties found the other application to periods, which is our motivation.

Let a smooth algebraic variety $X=\overline{X} - D$ over a field $K$ of characteristic $0$, where $\overline{X}$ is projective and smooth and $D$ is a normal crossings divisor. The de Rham realization $\pi_{1}^{\un,\dR}(X)$ of the pro-unipotent fundamental groupoid of $X$ is  the fundamental groupoid associated with the Tannakian category $\mathcal{C}^{\un,\dR}(X)$ of vector bundles with integrable connection with logarithmic singularity at $D$, and which are unipotent (\cite{Deligne}, \S10.27, \S10.30 (ii)).
We will denote by $Vec_{K}$ the category of vector spaces on $K$.

Each point $x$ of $X$ yields the fiber functor $\omega_{x} : C^{un,dR} \rightarrow Vec_{K}$ which sends a bundle to its fiber at $x$, thus defines a base-point of $\pi_{1}^{\un,\dR}(X)$ in the Tannakian sense. Similarly for each non-zero point $x$ of the tangent space at a point of $D$, in which case $\omega_{x}$ is called a tangential base-point of $\pi_{1}^{\un,\dR}(X)$.

For each couple of base-points $x,y$, the affine scheme $\pi_{1}^{\un,\dR}(X,y,x)$ is defined as the scheme of paths from $x$ to $y$ in the Tannakian sense, i.e. the scheme of isomorphisms of tensor functors $\omega_{x} \simeq \omega_{y}$.

For three base-points $x,y,z$, we have a canonical map 
$\pi_{1}^{\un,\dR}(X,z,y) \times \pi_{1}^{\un,\dR}(X,y,x) \rightarrow \pi_{1}^{\un,\dR}(X,z,x)$ defined by composing isomorphisms of fiber functors, called the groupoid multiplication. By this map, each $\pi_{1}^{\un,\dR}(X,x)= \pi_{1}^{\un,\dR}(X,x,x)$ is an affine group scheme, and each $\pi_{1}^{\un,\dR}(X,y,x)$ is a bi-torsor under the couple of algebraic groups $(\pi_{1}^{\un,\dR}(X,x),\pi_{1}^{\un,\dR}(X,y))$.

\subsubsection{Case of $\mathcal{M}_{0,4}$ and $\mathcal{M}_{0,5}$}

When $H^{1}(X,\mathcal{O}_{X})=0$, the functor which sends a bundle to its global sections is a tensor functor and thus defines a base-point $\omega_{\dR}$ of $\pi_{1}^{\un,\dR}(X)$, called the canonical base-point. For any base-point $x$, there is a canonical ismorphism $\omega \simeq \omega_{x}$. As a consequence, for all base-points $x,y$, there is an isomorphism of schemes 
$\pi_{1}^{\un,\dR}(X,y,x) \simeq \pi_{1}^{\un,\dR}(X,\omega_{\dR})$. These isomorphisms are compatible with the groupoid multiplication. (\cite{Deligne}, \S12.1, \S12.4). They enable us to do most of the computations in $\pi_{1}^{\un,\dR}(X,\omega_{\dR})$.

We will consider the cases $X=\mathcal{M}_{0,4}$ and $X=\mathcal{M}_{0,5}$, over a field $K$ of characteristic zero. We have $H^{1}(X,\mathcal{O}_{X})=0$, thus the canonical base-point $\omega_{\dR}$  of $\pi_{1}^{\un,\dR}(X)$ is defined. By \cite{Deligne}, \S12.8, we have an explicit description of $\pi_{1}^{\un,\dR}(X,\omega_{\dR})$ as a pro-unipotent algebraic group.

Namely, $\pi_{1}^{\un,\dR}(\mathcal{M}_{0,4},\omega_{\dR})$ is the exponential of the pro-nilpotent Lie algebra over $\mathbb{Q}$ defined by the generators $e_{0},e_{1},e_{\infty}$ and the relation $e_{0}+e_{1}+e_{\infty} = 0$, i.e. freely generated by $e_{0},e_{1}$. And  $\pi_{1}^{\un,\dR}(\mathcal{M}_{0,5},\omega_{\dR})$ is the exponential of the pro-nilpotent Lie algebra over $\mathbb{Q}$ defined by the generators $e_{i,j}, 1\leq i,j\leq 4$, and the relations $e_{ii} = 0$, $e_{ji} = e_{ij}$, $[e_{jk}+e_{jl},e_{kl}] = 0$ if $j,k,l$ are pairwise distinct, and $[e_{ij},e_{kl}] = 0$ if $i,j,k,l$ are pairwise distinct.

We will now describe in more detail the scheme  $\pi_{1}^{\un,\dR}(\mathcal{M}_{0,4},\omega_{\dR})$ which we will denote by $\Pi$, following the notation of \cite{Deligne Goncharov}, \S5.  Its Hopf algebra $\mathcal{O}^{\sh}$ is a particular case of a shuffle Hopf algebras. Generalities on shuffle Hopf algebras and their duals as universal enveloping algebras of free Lie algebras can be found in \cite{R}. Most of our computations will be done using $\Pi$.

\subsubsection{The Hopf algebra of $\Pi =\pi_{1}^{\un,\dR}(\mathcal{M}_{0,4},\omega_{\dR})$}

The affine Hopf algebra $\mathcal{O}^{\sh} = \mathcal{O}(\Pi)$ is the shuffle Hopf algebra over $\mathbb{Q}$ on the alphabet in two letters $\{e_{0},e_{1}\}$. A basis of $\mathcal{O}^{\sh}$ as a vector space is the set of words on $\{e_{0},e_{1}\}$ (including the empty word, denoted by $1$, and we take the convention to read words from the right to the left.)  The product $\sh$, called the \emph{shuffle product}, is defined by 
$e_{i_{n+n'}} \cdots e_{i_{n+1}} \text{ }\sh\text{ }e_{i_{n}}\cdots e_{i_{1}}
= \sum_{\sigma} e_{i_{\sigma^{-1}(n+n')}} \cdots e_{i_{\sigma^{-1}(1)}}$
where the sum is over the permutations $\sigma$ such that 
$\sigma(1) < \cdots<\sigma(n)$ and $\sigma(n+1)<\cdots<\sigma(n+n')$; the coproduct is defined by the deconcatenation $dec$ of words; the antipode is defined by $S : e_{i_{n}} \cdots e_{i_{1}} \mapsto (-1)^{n} e_{i_{1}} \cdots e_{i_{n}}$; the counit is the augmentation map.

For any two functions $a,b : \mathcal{O}^{\sh} \rightarrow A$, where $A$ is any $\mathbb{Q}$-algebra, their convolution product $a \bullet b$ is $\mult \circ (a \otimes b) \circ \dec$ where $\mult$ is the multiplication $A \otimes A \rightarrow A$.

\subsubsection{The completed dual of the Hopf algebra of $\Pi = \pi_{1}^{\un,\dR}(\mathcal{M}_{0,4},\omega_{\dR})$}

$\mathcal{O}(\Pi)$ is a graded Hopf algebra, where the grading is the number of letters of words on $\{e_{0},e_{1}\}$, called their weight. Its completed dual $\widehat{\mathcal{O}(\Pi)}^{\vee}$ as a graded Hopf algebra is the universal enveloping algebra of the completed (with respect to the lower central series) free Lie algebra $\mathbb{L}$ on the two generators $e_{0},e_{1}$. $\widehat{\mathcal{O}(\Pi)}^{\vee}$ is equal to $\mathbb{Q}\langle \langle e_{0},e_{1} \rangle\rangle$, the $\mathbb{Q}$-algebra of non-commutative power series over the two variables $e_{0},e_{1}$. Its multiplication is the usual multiplication of non-commutative formal power series. Its coproduct is the shuffle coproduct $\Delta_{\sh}$ defined by $\Delta_{\sh}(f)[w \otimes w'] = f[w\text{ }\sh\text{ }w']$, the multiplicative coproduct which satisfies $\Delta_{\sh}(e_{i}) = e_{i} \otimes 1 + 1 \otimes e_{i}$ for $i = 0,1$. It is usual to refer to the following description of $\Pi$ and its Lie algebra : for any $\mathbb{Q}$-algebra $A$,

(a) $\Pi(A)$ is the set of elements $f \in A \langle\langle e_{0},e_{1} \rangle\rangle$ which satisfy the shuffle equation, i.e. such that $f[1]=1$ and, for all non-empty words $w,w'$, we have $f[w\text{ }\sh\text{ }w'] = f[w]f[w']$,  and $f[1] = 1$. This amounts to say that $\Pi(A)$ is the set of grouplike elements in $A \langle\langle e_{0},e_{1}\rangle\rangle$, i.e. the elements $f$ such that $\Delta_{\sh}(f)= f \otimes f$.

(b) $\Lie \Pi(A)$ is the set of elements $f \in A \langle\langle e_{0},e_{1}\rangle\rangle$ which satisfy the shuffle equations modulo products, i.e. such that for all non-empty words $w,w'$, we have $f[w\text{ }\sh\text{ }w'] = 0$. This amounts to say that $\Lie \Pi(A)$ is the set of primitive elements in $A \langle\langle e_{0},e_{1} \rangle\rangle$, i.e. the elements such that $\Delta_{\sh}(f) = f \otimes 1 + 1 \otimes f$.

\subsubsection{A few more conventions}

We will often denote an element $f \in A \langle\langle e_{0},e_{1} \rangle\rangle$ as $f(e_{0},e_{1})$.

We usually write a word on $\{e_{0},e_{1}\}$ in the form $e_{0}^{n_{d}-1}e_{1} \cdots e_{0}^{n_{1}-1}e_{1}e_{0}^{n_{0}-1}$, where $d$ and the $n_{i}$'s are positive integers, or, if $d=0$, $e_{0}^{n_{0}-1}$ when $n_{0}$ is a positive integer. We will see that for most computations we can consider only words such that $n_{0}=1$ and $n_{d} \geq 2$. For short we will often denote a word $e_{0}^{n_{d}-1}e_{1} \cdots e_{0}^{n_{1}-1}e_{1}e_{0}^{n_{0}-1}$ by $\textbf{e}^{n_{d}-1,\ldots,n_{0}-1}$.

We denote by $\tilde{\Pi}$ the subgroup scheme of $\Pi$ defined by the equations $f[e_{0}]=f[e_{1}]=0$. It is also the exponential of $[\mathbb{L},\mathbb{L}]$, i.e. it is the commutator subgroup of $\Pi$.

\subsection{Associators}

From now on, $K$ is a field of characteristic zero (and we will later take it to be $\mathbb{C}$ or $\mathbb{Q}_{p}$ for a prime $p$).
We review a few definitions and properties from \cite{Drinfeld}. Let $\mu \in K$. Some of the equations below use the fact that, for $f=f(e_{0},e_{1}) \in \pi_{1}^{\un,\dR}(\mathcal{M}_{0,4},\omega_{\dR})(K)$, and for $i,j,k,l$ pairwise distinct, $f(e_{ij},e_{kl})$ is a point in 
$\pi_{1}^{\un,\dR}(\mathcal{M}_{0,5},\omega_{\dR})$.

\subsubsection{$M_{\mu}$}

(\cite{Drinfeld}, beginning of \S5 and just after (5.3)) The scheme of associators with parameter $\mu$ is the subscheme $M_{\mu}$ of $\tilde{\Pi}$ whose points are elements $\Phi$ satisfying the following equations in 
$\pi_{1}^{\un,\dR}(\mathcal{M}_{0,5},\omega_{\dR})$.

\begin{equation} \label{eq:pentagon} \Phi(e_{12},e_{23}+e_{24})
\Phi(e_{13}+e_{23},e_{34}) = 
\Phi(e_{23},e_{34}) \Phi(e_{12}+e_{13},e_{24}+e_{34})
\Phi(e_{12},e_{23}) ,
\end{equation}

\begin{equation} \label{eq:hexagon} e^{\frac{\mu e_{13}}{2}}\Phi(e_{13},e_{12}) 
e^{\frac{\mu e_{12}}{2}}\Phi(e_{31},e_{12})
e^{\frac{\mu e_{13}}{2}}\Phi(e_{23},e_{31}) =  e^{\frac{\mu(e_{12}+e_{23}+e_{31})}{2}} ,
\end{equation}
\begin{equation} \label{eq:duality} \Phi(e_{12},e_{23})\Phi(e_{23},e_{12}) = 1.
\end{equation}
\noindent Equations (\ref{eq:pentagon}), (\ref{eq:hexagon}) and (\ref{eq:duality}) are called, respectively, the pentagon, hexagon and duality equations. We have, if $\mu\not=0$, $\Phi(e_{0},e_{1}) \in M_{\mu}(K) \Leftrightarrow \Phi(\frac{1}{\mu}e_{0},\frac{1}{\mu}e_{1}) \in M_{1}(K)$.

\subsubsection{$\GRT_{1}$} (\cite{Drinfeld}, (5.12) to (5.15)) The graded Grothendieck-Teichmüller group is the subscheme $\GRT_{1}$ of $\tilde{\Pi}$ whose points are elements $\Phi$ satisfying the following equations in $\Pi$, resp. 
$\pi_{1}^{\un,\dR}(\mathcal{M}_{0,5},\omega_{\dR})$ :
\begin{equation} \label{eq:2 cycle} \Phi(e_{0},e_{1})\Phi(e_{1},e_{0}) = 1 ,
\end{equation}
\begin{equation} \label{eq:3 cycle} \Phi(e_{\infty},e_{0})\Phi(e_{1},e_{\infty})
\Phi(e_{0},e_{1}) = 1 ,
\end{equation}
\begin{equation} \label{eq:5 cycle}
\Phi(e_{12},e_{23}+e_{24}) \Phi(e_{13}+e_{23},e_{34})
= \Phi(e_{23},e_{34}) \Phi(e_{12}+e_{13},e_{24}+e_{34}) \Phi(e_{12},e_{23})  ,
\end{equation}
\begin{equation} \label{eq:CH0}
e_{0} + \Phi^{-1}(e_{0},e_{1})e_{1}\Phi(e_{0},e_{1}) + \Phi(e_{0},e_{\infty})^{-1}e_{\infty} \Phi(e_{0},e_{\infty}) = 0 .
\end{equation}
\noindent Equations (\ref{eq:2 cycle}), (\ref{eq:3 cycle}), (\ref{eq:5 cycle}), (\ref{eq:CH0}) are called, respectively the 2-cycle or duality, 3-cycle, 5-cycle or pentagon equation, and the equation of special automorphisms.

\subsubsection{Torsor structure} 

Equations (\ref{eq:2 cycle}), (\ref{eq:3 cycle}) and (\ref{eq:5 cycle}) imply (\ref{eq:CH0}), and $\GRT_{1} = M_{0}$  (\cite{Drinfeld}, Proposition 5.9).
$\GRT_{1}$ is a group scheme with the Ihara product defined by $(g_{2} \circ g_{1})(e_{0},e_{1}) = 
g_{2}(e_{0},e_{1})g_{1}(e_{0},g_{2}^{-1}e_{1}g_{2})$  (\cite{Drinfeld}, equation (5.16)), and the multiplication by the Ihara product defines an action of $\GRT_{1}$ on $M_{\mu}$ which makes $M_{\mu}$ into a $\GRT_{1}$-torsor (\cite{Drinfeld}, Proposition 5.5).
For each $\mu \in K$, an associator with parameter $\mu$ satisfies (\cite{AET}, \S5.2) :
\begin{equation} \label{eq:CH}
-\Phi^{-1}(e_{0},e_{1}) e^{-\mu e_{1}}\Phi(e_{0},e_{1}) e^{-\mu e_{0}}
= e^{\frac{\mu}{2}e_{0}} \Phi(e_{0},e_{\infty})^{-1} e^{\mu e_{\infty}} \Phi(e_{0},e_{\infty}) e^{-\frac{\mu}{2}e_{0}} .
\end{equation} Equation (\ref{eq:CH0}) is the coefficient of degree $1$ with respect to $\mu$ of equation (\ref{eq:CH}). We take the convention that $\mu$ has depth $0$ and weight $1$.

\subsubsection{Equations from $\mathcal{M}_{0,4}$\label{one dimensional}}

We will consider in the next parts equations expressible in terms of $\pi_{1}^{\un}(\mathcal{M}_{0,4})$, which we call one-dimensional equations. In the case of $\GRT_{1}$, these are equations (\ref{eq:2 cycle}), (\ref{eq:3 cycle}) and (\ref{eq:CH0}). In the case of $\M_{\mu}$, with $\mu\not= 0$, these are equations (\ref{eq:hexagon}) (\ref{eq:duality}), modulo the Lie ideal generated by $e_{12}+e_{23}+e_{31}$, and equation (\ref{eq:CH}). Modulo this ideal we will identify $e_{12},e_{23},e_{31}$ to $e_{0},e_{1},e_{\infty}$ respectively.

\section{Properties for computations with the de Rham fundamental groupoid of $\mathbb{P}^{1} - \{0,1,\infty\}$}

We write the compatibility between various operations on $\Pi = \Spec(\mathcal{O}^{\sh}) = \pi_{1}^{\un,\dR}(\mathbb{P}^{1} - \{0,1,\infty\},\omega_{dR})$ and the depth filtration.
 
\subsection{Shuffle algebra, weight and depth}

\subsubsection{Around the depth filtration and the weight}

Let $\mathcal{O}^{\sh}_{[0,d]}$ be the vector subspace of $\mathcal{O}^{\sh}$ generated by shuffle products of words of depth $\leq d$.
Let $\mathcal{O}^{\sh}_{n,[0,d]} \subset \mathcal{O}^{\sh}_{[0,d]}$ the subspace of elements of weight $n$.
\newline\indent We say that a linear map $f : \mathcal{O}^{\sh} \rightarrow \mathcal{O}^{\sh} $ preserves the depth filtration if we have $f (\mathcal{O}^{\sh}_{[0,d]}) \subset \mathcal{O}^{\sh}_{[0,d]}$ for any $d$. For such an $f$, we denote by $\gr_{D}(f)$ the map which associates to a word $w$ of depth $d$ the terms of depth $d$ in $f(w)$.
\newline\indent On the other hand, $\mathcal{O}^{\sh}$ is the graded algebra $\underset{n \geq 0}{\oplus} \mathcal{O}^{\sh}_{n}$ where $\mathcal{O}^{\sh}_{n}$ is the subspace of elements of weight $n$. Thus, let $R=(R_{n})_{n \geq 0}$ be an increasing ring filtration on $\mathbb{Q}$. We denote by ${}_R \mathcal{O}^{\sh}$ the restriction of scalars of $\mathcal{O}^{\sh}$ to $R$, i.e. for any $n\geq 0$, we allow scalars in $R_{n}$ for terms in $\mathcal{O}^{\sh}_{n}$. We will denote by $\mathbb{Z}$ the constant ring filtration defined by $R_{n}=\mathbb{Z}$ for all $n$.
\newline\indent Let $R$ such a filtration. Let $f \in \tilde{\Pi}(K)$, let $w \in \mathcal{O}^{\sh}$ be a word of depth $d$ and weight $n$. We say that $w$ admits a depth reduction for $f$ with coefficients in $R$ if $f[w]$ is in the $R_{n}$-module generated by elements $f[w'_{1}] \cdots f[w'_{i}]$ with words $w'_{j}$ such that $\depth(w'_{j})< d$ for all $j$ and $\sum_{j=1}^{i} \weight(w'_{j})=n$.

\subsubsection{A few consequences of the shuffle equations\label{convergent words}}

Let $f \in K \langle \langle e_{0},e_{1}\rangle\rangle$.
\newline If for all words $w$, $f[w\text{ }\sh\text{ }e_{0}] = 0$, (in particular, if the shuffle equation of (1.1) is satisfied), then, for any positive integers $d$, $n_{i}$ ($1 \leq i \leq d)$ and $l$, we have :
\begin{equation} \label{eq:shuffle1} f[e_{0}^{n_{d}-1}e_{1}\cdots e_{0}^{n_{1}-1}e_{1}e_{0}^{l}] = \sum_{\substack{l_{1},\ldots,l_{d}\geq 0 \\ l_{1}+\cdots+l_{d} = l}} \prod_{i=1}^{d} {-n_{i} \choose l_{i}} f[e_{0}^{n_{d}+l_{d}-1}e_{1} \cdots e_{0}^{n_{1}+l_{1}-1}e_{1}] .
\end{equation}
This can be proved by induction on $l$ by developing the equation 
$f[e_{0}^{n_{d}-1}e_{1}\cdots e_{0}^{n_{1}-1}e_{1}\text{ }\sh\text{ }e_{0}^{l}]=0$.

If, for any word $w$, $f[w\text{ }\sh\text{  }e_{1}] = 0$ (in particular, if the shuffle equation of (1.2) is satisfied), then, for any word $w$ and any positive integer $l$, we have :
\begin{equation} \label{eq:shuffle2}  f[e_{1}^{l}e_{0}w] =\frac{(-1)^{l}}{l!} f[e_{0}(e_{1}^{l}\text{ }\sh\text{ }w)].
\end{equation}

This can be proved by induction on $l$ by developing the equation $f[e_{1} \text{ }\sh\text{ }e_{1}^{l-1}e_{0}w]= 0$.

It follows from (\ref{eq:shuffle1}) and (\ref{eq:shuffle2}) that the vector space $f(\mathcal{O}^{\sh}_{n,[0,d]})$ is generated by the coefficients of the form $f[e_{0}^{n_{d}-1}e_{1} \cdots e_{0}^{n_{1}-1}e_{1}]$ with $n_{1}+\cdots+n_{d}=n$ and $n_{d} \geq 2$. This holds in particular for the elements $f \in \tilde{\Pi}(K)$ (they satisfy the shuffle equation in the sense of \S1.1.4 (a)) and for the elements $g^{-1}e_{x}g$ with $g \in \Pi(K)$ and $x \in \{0,1,\infty\}$ (they satisfy the shuffle equation modulo products in the sense of \S1.1.4 (b)).

\subsubsection{Derivations and the shuffle product}

One has derivations $\partial_{e_{0}},\partial_{e_{1}}, \tilde{\partial}_{e_{0}},\tilde{\partial}_{e_{1}},\partial,\tilde{\partial} : \mathcal{O}^{\sh} \rightarrow \mathcal{O}^{\sh}$ with respect to the shuffle product, defined by
$$ \forall\text{ }\text{words}\text{ }w,\text{ } \partial_{e_{i}}(e_{i}w) = w,\text{ } \partial_{e_{i}}(e_{1-i}w) = 0,\text{ } \partial_{e_{i}}(1) = 0 ,$$
$$ \forall\text{ }\text{words}\text{ }w,\text{ } \tilde{\partial}_{e_{i}}(we_{i}) = w,\text{ } \tilde{\partial}_{e_{i}}(we_{1-i}) = 0,\text{ } \tilde{\partial}_{e_{i}}(1) = 0 ,$$
$$ \partial = \partial_{e_{0}} + \partial_{e_{1}} , $$
$$ \tilde{\partial} =\partial_{e_{0}}+\tilde{\partial}_{e_{1}} . $$
The shuffle product is actually characterized by the fact that $\partial_{e_{0}}$ and $\partial_{e_{1}}$ (resp. $\tilde{\partial}_{e_{0}}$ and $\tilde{\partial}_{e_{1}}$) are derivations.

\subsection{A few operations}

\subsubsection{Inversion}

We define a map $\inv : \mathcal{O}^{\sh} \rightarrow \mathcal{O}^{\sh}$ by induction on the weight by $\inv(1)=1$, $\inv(e_{1})=\inv(e_{0})=0$, $(\inv \bullet \id)(w)= 0$, and, for any word $w$ of weight $\geq 2$,

\begin{equation} \label{eq:rec inv} \inv(w) + w + \sum_{\substack{w_{1}w_{2}=w \\ w_{1},w_{2} \not \in \{1\} \cup e_{1}^{\mathbb{N}}}} \inv(w_{1}) \sh w_{2} = 0 
\end{equation}

\begin{Fact} \label{fact depth graded inversion} for any $f \in \tilde{\Pi}(K)$,

(i) for all words $w$, we have $f^{-1}[w] = f[\inv(w)]$.

(ii) The map $inv$ preserves the depth filtration and we have, for any word $w$ of depth $>0$,
$$ \gr_{D}(\inv)(w) = - w. $$
\end{Fact}

Indeed, (i) is proved by writing the coefficients of $f^{-1}f=1$, and from the fact that we have, for all $n>0$, $f[e_{0}^{n}]=f[e_{1}^{n}] = f^{-1}[e_{0}^{n}]= f^{-1}[e_{1}^{n}] = 0$. And (ii) follows from (\ref{eq:rec inv}) by induction on the weight.

\subsubsection{Adjoint action}

Let $x \in \{0,1,\infty\}$ and $\mu \in K - \{0\}$. Consider the maps $\Pi(K) \rightarrow \Lie(\Pi)(K)$ and $\Pi(K) \rightarrow \Pi(K)$ defined as $\Ad(.)(e_{i}) : u \mapsto u^{-1}e_{i}u$ and $\Ad(.)(e^{\mu e_{i}}) : u \mapsto u^{-1}e^{\mu e_{i}}u$. (This notation comes from our convention to read the groupoid multiplication in $\pi_{1}^{\un,\dR}(\mathbb{P}^{1} - \{0,1,\infty\})$ from the right to the left.)

\begin{Lemma} \label{injectivity} The restrictions of $\Ad(.)(e_{i})$ and $\Ad(.)(e^{\mu e_{i}})$ to $\tilde{\Pi}(K)$ are injective.
\end{Lemma}

\begin{proof} It is sufficient to treat for example the case where $x=1$. The cases $x=0$ and $x=\infty$ are similar and can be deduced by applying the natural isomorphisms 
$K \langle \langle e_{1},e_{0}\rangle\rangle \simeq K\langle \langle e_{0},e_{1} \rangle \rangle$ and $K \langle \langle e_{0},e_{1}\rangle\rangle \simeq K\langle \langle e_{0},e_{\infty} \rangle \rangle$.
\newline\indent (i) Let $u \in K\langle\langle e_{0},e_{1} \rangle\rangle$ such that $u$ commutes to $e_{1}$. Let $w$ a word which is not of the form $e_{1}^{n}$, $n \geq 1$. It can be written in a unique way in the form $e_{1}^{a(w)}e_{0}z$, with $a(w) \geq 0$ and $z$ a word. We have $f[w] = (ue_{1})[we_{1}] = (e_{1}u)[we_{1}] = u(\partial_{e_{1}}(w)e_{1})$. By induction on $a(w)$, this shows that $f[w] = 0$ for all words $w$ containing at least one letter $e_{0}$. Thus $u \in  K\langle\langle e_{1} \rangle\rangle$. (See also \cite{Unver MZV}, \S5.3).
\newline\indent (ii) Let $u \in K \langle\langle e_{0},e_{1} \rangle\rangle$ such that $u$ commutes to $e^{\mu e_{1}}$. Let a word $w$ which contains at least one letter $e_{0}$. It can be written in a unique way in the form $e_{1}^{a(w)} z e_{1}^{b(w)}$ with $a(w),b(w) \geq 0$ and $z$ a word such that $\partial_{e_{1}}(z) = \tilde{\partial}_{e_{1}}(z) = 0$ (i.e. $z=e_{0}$ or $z$ is of the form $e_{0} \cdots e_{0}$). We have $(e^{\mu e_{1}}-1)u = u(e^{\mu e_{1}}-1)$, whence $\sum\limits_{l=1}^{a(w)} \frac{\mu ^{l}}{l!} f[e_{1}^{a(w)-l} z e_{1}^{b(w)}] = 
\sum\limits_{l'=1}^{b(w)} \frac{\mu^{l'}}{l'!} f[e_{1}^{a(w)} z e_{1}^{b(w)-l'}]$. By induction on $(a(w),b(w))$ with the lexicographical order, this shows that $f[w] = 0$ for all words $w$ containing at least one letter $e_{0}$, and the end of the proof is similar to 1). Thus $u \in  K\langle\langle e_{1} \rangle\rangle$.
\newline\indent (iii) The elements of $K \langle\langle e_{1}\rangle\rangle$ which satisfy the shuffle equation are those of the form $\exp(\lambda e_{1})$ with $\lambda \in K$. The only element of that type in $\tilde{\Pi}(K)$ is $1$.
\end{proof}

\subsubsection{Automorphisms induced by homographies}

Here $X=\mathbb{P}^{1} - \{0,1,\infty\}$. The Tannakian category $\mathcal{C}^{\un,\dR}(X)$ which is subjacent to $\pi_{1}^{\un,\dR}(X)$ (see \S1.1.1) has a initial pro-object : 
the trivial bundle $\Pi \times X$ endowed with the Knizhnik-Zamolodchikov connexion $\displaystyle \nabla_{\KZ} : f \mapsto df - (\frac{dz}{z}e_{0} + \frac{dz}{z-1}e_{1})$.

The group $\Aut(\mathbb{P}^{1} - \{0,1,\infty\})$ is the set homographies of $\mathbb{P}^{1}$ which induce a permutation of $\{0,1,\infty\}$, and it is thus isomorphic to the group of permutations $S_{3}$. For each $\sigma$ in $\Aut(\mathbb{P}^{1} - \{0,1,\infty\})$, the functoriality of $\pi_{1}^{\un,\dR}$ induces the automorphism $\sigma_{\ast}$ of $\Pi$, defined by $f(e_{0},e_{1}) \mapsto f(e_{\sigma(0)},e_{\sigma(1)})$. Indeed, this automorphism is characterized by the fact that it induces an automorphism of the bundle $\Pi \times X$ which commutes with $\nabla_{\KZ}$, i.e. we have $\displaystyle \frac{d\sigma(z)}{\sigma(z)}e_{0} + \frac{d\sigma(z)}{\sigma(z)-1}e_{1} = \frac{dz}{z}e_{\sigma(0)} + \frac{dz}{z-1}e_{\sigma(1)}$.

It also induces an automorphism $\sigma_{\ast}^{\vee}$ of $\mathcal{O}^{\sh}=\mathcal{O}(\Pi)$. Below we consider the case of $z \mapsto \frac{1}{z}$.

\begin{Lemma} \label{coeff automorphism} The map $(z \mapsto \frac{1}{z})_{\ast}^{\vee}$ is preserves the depth filtration and we have, for any word $w$,
$$ \gr_{D} \bigg(z \mapsto \frac{1}{z}\bigg)_{\ast}^{\vee} (w) = (-1)^{\weight(w)-\depth(w)} w .$$
\end{Lemma}

\begin{proof} The automorphism  $(z \mapsto \frac{1}{z})_{\ast}^{\vee}$ is $f(e_{0},e_{1}) \in \Pi(K) \mapsto f(e_{\infty},e_{1}) \in \Pi(K)$. It corresponds to an automorphism of $\tilde{O}^{\sh}$, given by $w(e_{0},e_{1}) \mapsto w(-e_{0},-e_{0}+e_{1})$. When expanding $w(-e_{0},-e_{0}+e_{1})$ as a linear combination of words, the highest depth term is 
$w(-e_{0},e_{1}) = (-1)^{\weight(w) - \depth(w)}w(e_{0},e_{1})$.
\end{proof}

\begin{Notation} For any $f =f(e_{0},e_{1}) \in \tilde{\Pi}(K)$, we denote by $f_{\infty}=f(e_{0},e_{\infty}) = (z \mapsto \frac{z}{z-1})_{\ast}(f)$.
\end{Notation}

In the next sections, the idea of the proof is to reformulate certain associator equations as an equality between modules of coefficients of $\Phi$ and of $\Phi_{\infty}$. These are the ``redundancies'' in this equality which will give the depth reductions.

\section{Depth reductions for degenerated associators}

We prove the main theorem for points of $\GRT_{1}=M_{0}$, i.e degenerated associators.

\subsection{Elimination of the duality equation}

In the one-dimensional equations of $\GRT_{1}$ in the sense of \S1.2.4, we eliminate equation (\ref{eq:2 cycle}).

\begin{Lemma} \label{proposition} An element 
$f \in \tilde{\Pi}(K)$ satisfies equations 
(\ref{eq:2 cycle}), (\ref{eq:3 cycle}) and (\ref{eq:CH0}) if and only if :
\begin{equation} \label{eq:hexagon part mu equal 0}f(e_{0},e_{\infty}) = f(e_{\infty},e_{1})^{-1} f(e_{0},e_{1}) ,
\end{equation}
\begin{equation} \label{eq:CH0 mod}
- e_{0} - f^{-1}(e_{0},e_{1})e_{1}f(e_{0},e_{1}) = f(e_{0},e_{\infty})^{-1}e_{\infty} f(e_{0},e_{\infty}) .
\end{equation}
\end{Lemma}

\begin{proof} Let us apply to (\ref{eq:CH0}), on the one hand, the conjugation by $f^{-1}$, and, on the other hand, the change of variables $(e_{0},e_{1}) \rightarrow (e_{1},e_{0})$. We obtain the following system of two equations
$$ f(e_{0},e_{1}) e_{0} f(e_{0},e_{1})^{-1} + e_{1} + f(e_{0},e_{1})f(e_{0},e_{\infty})^{-1}e_{\infty}f(e_{0},e_{\infty})f(e_{0},e_{1})^{-1} = 0 , $$
$$ e_{1} + f(e_{1},e_{0})e_{0}f(e_{1},e_{0})^{-1} + f(e_{1},e_{\infty})^{-1}e_{\infty} f(e_{1},e_{\infty}) = 0 . $$
On the other hand, by Lemma \ref{injectivity}, $f$ satisfies (\ref{eq:duality}) if and only if $(f^{-1}e_{1}f)(e_{1},e_{0})= f e_{0}f^{-1}$. Moreover, by Lemma \ref{injectivity}, the map $g \mapsto g^{-1}e_{\infty}g$ is injective. This proves that, assuming equation (\ref{eq:CH0}), $f$ satisfies the duality equation (\ref{eq:2 cycle}) if and only if $f$ satisfies the 3-cycle equation (\ref{eq:3 cycle}).
\end{proof}

We denote by $\delta_{e_{x}} : \mathcal{O}^{\sh} \rightarrow \mathcal{O}^{\sh}$ the linear map which sends a word $w$ to $e_{x}$ if $w=e_{x}$ and to $0$ otherwise. 

We define $l_{0},\tilde{l}_{0},\tilde{l}_{0}^{\infty} : \mathcal{O}^{\sh} \rightarrow \mathcal{O}^{\sh}$ as follows 
$$ l_{0} = -(\inv \circ (z \mapsto \frac{z}{z-1})_{\ast}^{\vee}) \bullet \id $$
$$ \tilde{l}_{0} = - (\inv \bullet \delta_{e_{1}} \bullet \id) - \delta_{e_{0}} $$
$$ \tilde{l}_{\infty} = \inv \bullet ( -\delta_{e_{0}} - \delta_{e_{1}}) \bullet \id $$
By the definitions, for $f,g \in \tilde{\Pi}(K)$, regarded as maps $\mathcal{O}^{\sh} \rightarrow K$, we have 
$- f(e_{\infty},e_{1})^{-1}f(e_{0},e_{1}) = f(e_{0},e_{1}) \circ l_{0}$, 
$- f(e_{0},e_{1})^{-1}e_{1}f(e_{0},e_{1}) - e_{0} = f(e_{0},e_{1}) \circ \tilde{l}_{0}$, 
and 
$g(e_{0},e_{1})^{-1}e_{\infty}g(e_{0},e_{1})  = g(e_{0},e_{1}) \circ \tilde{l}_{0}^{\infty}$. 

\noindent In particular, equations (\ref{eq:CH0 mod}) and (\ref{eq:hexagon part mu equal 0}) are respectively equivalent to
\begin{equation} \label{eq:summary of equations avec tilde}
f \circ \tilde{l}_{0} = f_{\infty} \circ \tilde{l}_{\infty}, 
\end{equation}
\begin{equation} \label{eq:summary of equations sans tilde}
f \circ l_{0} = f_{\infty} .
\end{equation}

\subsection{A relation between $\Phi$ and $\Phi_{\infty}$}

We now reformulate equations (\ref{eq:summary of equations avec tilde}) and (\ref{eq:summary of equations sans tilde}) as an equality between modules of coefficients of $\Phi$ and $\Phi_{\infty}$, compatible with the depth filtration, and with coefficients in $\mathbb{Z}$.

\begin{Proposition} \label{main theorem 0}
Let $\Phi \in \tilde{\Pi}(K)$ satisfying (\ref{eq:summary of equations avec tilde}) and (\ref{eq:summary of equations sans tilde}). Then we have, for any non-negative integer $d$ :
\begin{equation} \label{eq:equality of modules 0} \Phi \big({}_\mathbb{Z} \mathcal{O}^{\sh}_{[0,d]} \big)
= \Phi_{\infty} \big({}_\mathbb{Z} \mathcal{O}^{\sh}_{[0,d+1]} \big) .
\end{equation}
\noindent More precisely, for any positive integers $d$ and $n_{i}$ ($1 \leq i \leq d$) such that $n_{d}\geq 2$, resp. $d$ and $n_{i}$ ($1 \leq i \leq d+1$) such that $n_{d+1}\geq 2$, one has the following congruences :
\begin{equation} \label{eq:congruence 3 mu = 0}
\Phi[\textbf{e}^{n_{d}-1,\ldots,n_{1}-1,0}] \equiv \Phi_{\infty}[\textbf{e}^{n_{d}-2,\ldots,n_{1}-1,0,0}] \mod \Phi({}_\mathbb{Z} \mathcal{O}^{\sh}_{[0,d-1]}),
\end{equation}
\begin{equation} \label{eq:congruence 4 mu = 0}
\Phi_{\infty}[\textbf{e}^{n_{d+1}-2,\ldots, n_{1}-1,0}] \equiv -\Phi[\textbf{e}^{n_{d+1}-1,\ldots,n_{1}-1}] \mod \Phi({}_\mathbb{Z} \mathcal{O}^{\sh}_{[0,d-1]}) .
\end{equation}
\noindent Furthermore,
\begin{equation} \label{eq:congruence 12 0}
(1 - (-1)^{\sum_{i=1}^{d}n_{i}-d}) \Phi[\textbf{e}^{n_{d}-1,\ldots,n_{1}-1,0}] \equiv \Phi_{\infty}[\textbf{e}^{n_{d}-1,\ldots,n_{1}-1,0}]
\mod \Phi({}_\mathbb{Z} \mathcal{O}^{\sh}_{[0,d-2]}) .
\end{equation}
\end{Proposition}

\begin{proof} (i) \label{lemma dg l3}
By Lemma \ref{fact depth graded inversion},
$\tilde{l}$ preserves the depth filtration and we have $\gr_{D}\tilde{l}(w) = 0$.
\noindent Let us assume that $w=\textbf{e}^{n_{d}-1,\ldots,n_{2}-1,0,0}$ with $n_{d} \geq 2$ ; then, more precisely, we have,  $\tilde{l}(w) \equiv - \tilde{\partial}w \mod {}_\mathbb{Z} \mathcal{O}^{\sh}_{[0,d-2]}$.
\newline\indent (ii) \label{lemma dg l4}
By Lemma \ref{fact depth graded inversion}, $\tilde{l}^{\infty}_{0}$ preserves the depth filtration and we have (with $w=\textbf{e}^{n_{d}-1,\ldots,n_{1}-1,n_{0}-1}$) :
\newline $\gr_{D}\tilde{l}^{\infty}_{0}(w) =  - 1_{n_{d}\geq 2} \partial w + 1_{n_{0}\geq 2} \tilde{\partial}w$. In particular, if $n_{d}\geq 2$ and $n_{0}=1$ we have : $\gr_{D}\tilde{l}^{\infty}_{0}(w) =  - \partial w$.
\newline\indent (iii) Let us prove the equality of modules (\ref{eq:equality of modules 0}) and the congruences (\ref{eq:congruence 3 mu = 0}) and (\ref{eq:congruence 4 mu = 0}) by induction on $d$. 
\newline\indent We first prove the result for $d=0$. By $\Phi \in \tilde{\Pi}(K)$ we have $\Phi[e_{0}]=\Phi[e_{1}]=0$. By the definition of $\Phi_{\infty}$ (Notation 2.4), this implies $\Phi_{\infty}[e_{0}]=\Phi_{\infty}[e_{1}]=0$. By the shuffle equation (1.1), we deduce $\Phi[e_{0}^{n}] = \Phi_{\infty}[e_{0}^{n}]=0$ for any positive integer $n$. We apply equation (\ref{eq:summary of equations avec tilde}) to the word $\text{e}^{n-1,0}$ with $n\geq 2$. We have
$(\Phi_{\infty}^{-1}e_{\infty}\Phi_{\infty})[\text{e}^{n-1,0}]=-\Phi_{\infty}[\text{e}^{n-2,0}]$ and $(e_{0}+\Phi^{-1}e_{1}\Phi)[\text{e}^{n-1,0}]=0$, whence $\Phi_{\infty}[\text{e}^{n-2,0}]=0$. By the shuffle equation (1.1), we deduce $\Phi_{\infty}({}_\mathbb{Z} \mathcal{O}^{\sh}_{[0,1]})= \Phi({}_\mathbb{Z} \mathcal{O}^{\sh}_{[0,0]})= 0$.
\newline\indent Let us assume the result for $d-1$.  Equation (\ref{eq:summary of equations avec tilde}) applied to a word $\textbf{e}^{n_{d+1}-1,\ldots,n_{2}-1,0,0}$ with $n_{d+1}\geq 2$, gives, by (i) and (ii),
$$ \Phi[\textbf{e}^{n_{d+1}-1,\ldots,n_{2}-1,0}] \in \Phi_{\infty}[\textbf{e}^{n_{d+1}-2,\ldots,n_{2}-1,0,0}] + \Phi_{\infty}({}_\mathbb{Z} \mathcal{O}^{\sh}_{[0,d]}) + \Phi({}_\mathbb{Z} \mathcal{O}^{\sh}_{[0,d-1]}); $$  by the induction hypothesis,
$\Phi({}_\mathbb{Z} \mathcal{O}^{\sh}_{[0,d-1]}) \subset \Phi_{\infty}({}_\mathbb{Z} \mathcal{O}^{\sh}_{[0,d]})$ ; by definition $\Phi_{\infty}({}_\mathbb{Z}  \mathcal{O}^{\sh}_{[0,d]}) \subset \Phi_{\infty}({}_\mathbb{Z} \mathcal{O}^{\sh}_{[0,d+1]} )$ ; whence $\Phi[\textbf{e}_{0}^{n_{d}-1,\ldots,n_{2}-1,0}] \in \Phi_{\infty}({}_\mathbb{Z} \mathcal{O}^{\sh}_{[0,d+1]})$ ; by \S \ref{convergent words}, this implies the inclusion $\Phi({}_\mathbb{Z} \mathcal{O}^{\sh}_{[0,d]}) \subset \Phi_{\infty}({}_\mathbb{Z} \mathcal{O}^{\sh}_{[0,d+1]})$; this also implies the congruence (\ref{eq:congruence 3 mu = 0}).
\newline\indent By equation (\ref{eq:summary of equations avec tilde}) applied to a word $\textbf{e}^{n_{d}-1,\ldots,n_{1}-1,0}$ with $n_{d}\geq 2$ and (i) and (ii), we have
$$ \Phi_{\infty}\big[ \textbf{e}^{n_{d+1}-2,n_{d}-1,\ldots ,n_{1}-1,0} \big]
\in -\Phi[\textbf{e}^{n_{d+1}-1,\ldots,n_{1}-1}] + \Phi_{\infty}({}_\mathbb{Z} \mathcal{O}^{\sh}_{[0,d]}) + \Phi({}_\mathbb{Z} \mathcal{O}^{\sh}_{[0,d-1]}); $$ 
by the induction hypothesis $\Phi_{\infty}({}_\mathbb{Z} \mathcal{O}^{\sh}_{[0,d]}) \subset \Phi({}_\mathbb{Z} \mathcal{O}^{\sh}_{[0,d-1]})$ ; by definition, $\Phi({}_\mathbb{Z} \mathcal{O}^{\sh}_{[0,d-1]}) \subset \Phi({}_\mathbb{Z} \mathcal{O}^{\sh}_{[0,d]})$, whence $ \Phi_{\infty}\big[ \textbf{e}^{n_{d}-1,n_{d-1}-1,\ldots ,n_{1}-1,0} \big] \in \Phi({}_\mathbb{Z} \mathcal{O}^{\sh}_{[0,d]})$ ; by \S\ref{convergent words}, this implies the converse inclusion $\Phi_{\infty}({}_\mathbb{Z} \mathcal{O}^{\sh}_{[0,d+1]}) \subset \Phi({}_\mathbb{Z} \mathcal{O}^{\sh}_{[0,d]})$ ; this also implies equation (\ref{eq:congruence 4 mu = 0}).
\newline\indent (iv) By Lemma \ref{fact depth graded inversion} and Lemma \ref{coeff automorphism}, $l_{0}$ preserves the depth filtration and we have
$\gr_{D}l_{0}(w) =  (1 - (-1)^{\weight(w)-\depth(w)})w$ 
\newline\indent (v) By (ii), (iv) and (\ref{eq:equality of modules 0}), equation (\ref{eq:congruence 12 0}) is the depth-graded version of equation (\ref{eq:summary of equations sans tilde}).
\end{proof}

\begin{Remark} By the above proof, one has full equalities beyond the congruences (\ref{eq:congruence 3 mu = 0}) (\ref{eq:congruence 4 mu = 0}) (\ref{eq:congruence 12 0}) which imply equations (\ref{eq:summary of equations sans tilde}) and (\ref{eq:summary of equations avec tilde}). By Lemma \ref{proposition}, they imply the one-dimensional part of the equations of $\GRT_{1}$. Thus we have reformulated the one-dimensional equations of $\GRT_{1}$ as an equality between modules of coefficients of $\Phi$ and $\Phi_{\infty}$.
\end{Remark}

\subsection{The adjoint depth reduction}

We deduce from Proposition \ref{main theorem 0} the depth reduction (i) of the Theorem for $\GRT_{1}$.

\begin{Corollary} \label{depth red f adjoint 0} Assume that $\Phi$ satisfies equations (\ref{eq:summary of equations avec tilde}), (\ref{eq:summary of equations sans tilde}). Then  $(n_{d},\ldots,n_{1}) + (-1)^{n_{d}+\cdots+n_{1}}(n_{1},\ldots,n_{d})$ admits a depth reduction for $\Phi$ with coefficients in $\mathbb{Z}$. More precisely :
\begin{multline}
\label{eq:formula for depth red adjoint}
\Phi[\textbf{e}^{n_{d}-1,\ldots,n_{1}-1,0}] + (-1)^{n_{d}+\cdots+n_{1}}\Phi[\textbf{e}^{n_{1}-1,\ldots,n_{d}-1,0}]
\\ \equiv 
 (-1)^{\sum\limits_{i=1}^{d} n_{i}} 
\sum_{\substack{l_{1},\ldots,l_{d-1} \geq 0
\\ 
l_{1}+\cdots+l_{d-1} = n_{d}}}
\prod_{i=1}^{d} {-n_{i} \choose l_{i}} \Phi[\textbf{e}^{n_{1}+l_{1}-1,\ldots,n_{d-1}+l_{d-1}-1,0}]
\\ + \sum_{\substack{l'_{2},\ldots,l'_{d} \geq 0
\\ l'_{2}+\cdots+l'_{d} = n_{1}}}
\prod_{i=2}^{d} {-n_{i} \choose l_{i}} \Phi[\textbf{e}^{n_{d}+l_{d}-1,\ldots, n_{2}+l_{2}-1,0}]
\mod \Phi({}_\mathbb{Z} \mathcal{O}^{\sh}_{[0,d-2]}) .
\end{multline}
Denoting by $z = \textbf{e}^{n_{d}-1,\ldots,n_{1}-1}$, the right-hand side in (\ref{eq:formula for depth red adjoint}) is also equal to $\Phi[ze_{0}] + \Phi^{-1}[e_{0}z] 
\mod \Phi({}_\mathbb{Z} \mathcal{O}^{\sh}_{[0,d-2]})$ and to $\Phi[ze_{0}] - \Phi[e_{0}z] 
\mod \Phi({}_\mathbb{Z} \mathcal{O}^{\sh}_{[0,d-2]})$.
\end{Corollary}

\begin{proof} We specialize equation (\ref{eq:summary of equations avec tilde}) to the coefficients of the form  $\textbf{e}^{0,n_{d}-1,\ldots,n_{1}-1,0}$. We obtain 
$$ (\Phi^{-1}e_{1}\Phi)[\textbf{e}^{0,n_{d}-1,\ldots,n_{1}-1,0}] \in \Phi_{\infty}^{-1}[\textbf{e}^{0,n_{d}-1,\ldots,n_{1}-1}]  +
\Phi_{\infty}[\textbf{e}^{n_{d}-1,\ldots,n_{1}-1,0}] +
\Phi_{\infty}({}_\mathbb{Z} \mathcal{O}^{\sh}_{[0,d-1]}) . $$
Moreover, we have $\Phi_{\infty}^{-1}[\textbf{e}^{0,n_{d}-1,\ldots,n_{1}-1}]=(-1)^{\sum_{i=1}^{d}n_{i}} \Phi_{\infty}[\textbf{e}^{n_{1}-1,\ldots,n_{d}-1,0}]$ by the formula for the antipode of $\mathcal{O}^{\sh}$ (\S1.1.2) and, on the other hand,
$$ (\Phi^{-1}e_{1}\Phi)[\textbf{e}^{0,n_{d}-1,\ldots,n_{1}-1,0}] \in  \Phi[\textbf{e}^{n_{d}-1,\ldots,n_{1}-1,0}] + (-1)^{n_{d}+\cdots+n_{1}}\Phi[\textbf{e}^{n_{1}-1,\ldots,n_{d}-1,0}] + \Phi({}_\mathbb{Z} \mathcal{O}^{\sh}_{[0,d-1]}) ; $$
we express 
$\Phi_{\infty}[\textbf{e}^{n_{d}-1,\ldots,n_{1}-1,0}] + (-1)^{\sum_{i=1}^{d}n_{i}} \Phi_{\infty}[\textbf{e}^{n_{1}-1,\ldots,n_{d}-1,0}]$ in terms of $\Phi$ in depth $\leq d-1$ by using (\ref{eq:congruence 4 mu = 0}) and then equation (\ref{eq:shuffle1}).
\end{proof}

In the case where $\Phi$ is the generating series of $p$MZV's, the numbers $(\Phi^{-1}e_{1}\Phi)[\textbf{e}^{0,n_{d}-1,\ldots,n_{1}-1,0}]$ are a particular case of adjoint $p$MZV's in the sense of \cite{J4}, Definition 1.2.1, whence our terminology adjoint depth reduction.

\subsection{The parity depth reduction}

We deduce from Proposition \ref{main theorem} the depth reduction (ii) of the Theorem for $\GRT_{1}$.

\begin{Corollary} \label{corollary parity depth reduction} Let $\Phi \in \tilde{\Pi}(K)$ satisfying equations (\ref{eq:summary of equations avec tilde}), (\ref{eq:summary of equations sans tilde}). Then, if $n_{1}+\cdots+n_{d}-d$ is odd,  
$(n_{1},\ldots,n_{d})$ admits a depth reduction for $\Phi$, with coefficients in $\mathbb{Z}$.
\end{Corollary}

\begin{proof} This is a consequence of the congruence (\ref{eq:congruence 12 0}) and the equality of modules (\ref{eq:equality of modules 0}).
\end{proof}

By combining the adjoint depth reduction and the parity depth reduction, we deduce :

\begin{Corollary} Let $\Phi \in \tilde{\Pi}(K)$ satisfying equations (\ref{eq:summary of equations avec tilde}), (\ref{eq:summary of equations sans tilde}). Then, if $n_{1}+\cdots+n_{d}-d$ is odd,  $(n_{d},\ldots,n_{1}) + (-1)^{n_{d}+\cdots+n_{1}}(n_{1},\ldots,n_{d})$ admits a depth reduction down to $d-2$ for $\Phi$ with coefficients in $\mathbb{Z}$.
\end{Corollary}

\begin{proof} This is a consequence of Corollary \ref{depth red f adjoint 0} and Corollary \ref{corollary parity depth reduction}.
\end{proof}

\begin{Example} \label{example depth red adjoint}
In depth one and two : for any positive integers $n$, $n_{1}$, $n_{2}$,
\begin{equation} \label{eq:finite 0} (\Phi^{-1}e_{1}\Phi)[\textbf{e}^{0,n-1,0}] =0 ,
\end{equation}
\begin{equation} \label{eq:finite 1} (\Phi^{-1}e_{1}\Phi)[\textbf{e}^{0,n_{2}-1,n_{1}-1,0}] \equiv (-1)^{n_{1}}{n_{1}+n_{2} \choose n_{1}} \Phi[\textbf{e}^{n_{1}+n_{2}-1,0}] .
\end{equation}
and if $n_{1}+n_{2}$ is even then $\Phi[\textbf{e}^{n_{1}+n_{2}-1,0}]=0$.
\end{Example}

\section{Depth reductions for non-degenerated associators}

We prove an analogue of the results of \S3 to the scheme of associators $M_{\mu}$ with $\mu\not=0$. This finishes the proof of the main theorem.

\subsection{Preliminaries}

We start with the following one-dimensional associator equations, in the sense of \S1.2.4 :

\begin{equation} \label{eq:CH mod}
f^{-1}(e_{0},e_{1})e^{-\mu e_{1}}f(e_{0},e_{1})e^{-\mu e_{0}} 
= e^{\frac{\mu}{2}e_{0}} f(e_{0},e_{\infty})^{-1}e^{\mu  e_{\infty}} f(e_{0},e_{\infty})e^{-\frac{\mu}{2}e_{0}} ;
\end{equation}
\begin{equation} \label{eq:hexagon part mu not 0} f(e_{\infty},e_{1})^{-1} e^{\frac{\mu}{2}e_{1}} f(e_{0},e_{1}) e^{\frac{\mu }{2}e_{0}} = e^{\frac{\mu}{2}e_{\infty}} f(e_{0},e_{\infty}) .
\end{equation}

We define morphisms $\mathcal{O}^{\sh} \rightarrow \mathcal{O}^{\sh}$, which correspond to morphisms of affine schemes $\Pi \rightarrow \Pi$, as follows. 

For $x \in \{0,1,\infty\}$, and $\lambda \in K$ we write $e^{\lambda x}$ for the map $\mathcal{O}^{\sh} \rightarrow \mathcal{O}^{\sh}$ which sends a word $w=x^{n}$ to $\frac{\lambda^{n}}{n!} x^{n}$, for all $n \geq 0$, and all other words to $0$.

$$ l_{\mu} =  \big( (z \mapsto \frac{1}{z})_{\ast}^{\vee}  \circ \inv \big)  \bullet e^{\frac{\mu}{2}e_{1}} \bullet \id \bullet e^{\frac{\mu}{2}e_{0}}  , $$
$$ l_{\mu }^{\infty} = e^{\frac{\mu}{2}e_{\infty}} \bullet \id $$
$$ \tilde{l}_{\mu} =   \inv \bullet e^{-\mu e_{1}} \bullet \id \bullet e^{-\mu e_{0}} , $$ 
$$ \tilde{l}_{\mu}^{\infty} = e^{\frac{\mu}{2}e_{0}} \bullet \inv \bullet e^{\mu e_{\infty}} \bullet \id \bullet e^{-\frac{\mu}{2}e_{0}} . $$
 
With these definitions, and viewing $f,f_{\infty}$ as maps $\mathcal{O}^{\sh} \rightarrow K$, equations  (\ref{eq:CH mod}) and (\ref{eq:hexagon part mu not 0}) are respectively equivalent to 
\begin{equation} \label{eq:summary of equations avec tilde mu}
f \circ \tilde{l}_{\mu} = f_{\infty} \circ \tilde{l}_{\mu}^{\infty} ,
\end{equation}
\begin{equation} \label{eq:summary of equations sans tilde mu}
f \circ l_{\mu} = f_{\infty} \circ l_{\mu}^{\infty} .
\end{equation}

\subsection{A relation between $\Phi$ and $\Phi_{\infty}$}

We reformulate equations (\ref{eq:summary of equations avec tilde mu}) and (\ref{eq:summary of equations sans tilde mu}) as a comparison between the modules of coefficients of $\Phi$ and $\Phi_{\infty}$, compatible with the depth filtration.
\newline In order to keep track of the denominators of the rational coefficients, let the ring filtration $R=(R_{n})_{n\geq 0}$ on $\mathbb{Q}[\mu]$ defined by $R_{n} = \sum\limits_{\substack{n_{1}+n_{2}\geq 0 \\ n_{1}+n_{2}=n}} \frac{1}{n_{1}!} \big(\frac{\mu}{2}\big)^{n_{2}}\mathbb{Z}$.

\begin{Proposition} \label{main theorem}
	Let $\Phi \in \tilde{\Pi}(K)$ satisfying equations (\ref{eq:summary of equations avec tilde mu}), (\ref{eq:summary of equations sans tilde mu}). Then we have, for any non-negative integer $d$ :
	\begin{equation} \label{eq:equality of modules} \Phi \big( {}_R \mathcal{O}^{\sh}_{[0,d]}  \big)
	= \Phi_{\infty} \big( {}_R \mathcal{O}^{\sh}_{[0,d+1]}   \big) .
	\end{equation}
	\noindent More precisely, for any positive integers $d$ and $n_{i}$ ($1 \leq i \leq d$) such that $n_{d}\geq 2$, resp. $n_{i}$ ($1 \leq i \leq d+1$) such that $n_{d+1}\geq 2$, one has the following congruences :
	\begin{multline} \label{eq:congruence 3 mu not 0}
	-\Phi[\textbf{e}^{n_{d}-1,\ldots,n_{1}-1,0}] 
	\equiv
	\\ \sum_{0 \leq n'_{d} \leq n_{d}-1}
	\frac{1}{n'_{d}!} \bigg(\frac{\mu}{2}\bigg)^{n'_{d}} \bigg( 
	\sum_{1 \leq l \leq n_{d}-1-n'_{d}}\frac{\mu^{l-1}}{l!} \text{ } \Phi_{\infty}[\textbf{e}^{n_{d}-1-l-n'_{d},\ldots,n_{1}-1,0,0}] - \Phi_{\infty}[\textbf{e}^{n_{d}-1-n'_{d},\ldots,n_{1}-1,0,0}] \bigg) \mod \Phi({}_R \mathcal{O}^{\sh}_{[0,d-1]}) ,
	\end{multline}
	\begin{multline} \label{eq:congruence 4 mu not 0}
	\sum_{\substack{0 \leq n'_{d+1} \leq n_{d+1}-1 \\ 1 \leq l \leq n_{d+1}-1-n'_{d+1}}} 
	\frac{1}{n'_{d+1}!} \bigg(\frac{\mu}{2}\bigg)^{n'_{d+1}}
	\frac{\mu^{l-1}}{l!} \text{ } \Phi_{\infty}[\textbf{e}^{n_{d+1}-1-l-n'_{d+1},\ldots ,n_{1}-1,0}]
	\equiv -\Phi[\textbf{e}^{n_{d+1}-1,\ldots,n_{1}-1}] 
	\mod \Phi({}_R \mathcal{O}^{\sh}_{[0,d-1]}) .
	\end{multline}
	Furthermore,
	\begin{multline} \label{eq:congruence 12}
	(1 - (-1)^{\sum_{i=1}^{d}n_{i}-d}) \Phi[\textbf{e}^{n_{d}-1,\ldots,n_{1}-1,0}] 
	\equiv \sum_{n'=0}^{n_{d}-1} \frac{1}{n'!} \bigg( \frac{\mu}{2} \bigg)^{n'} \Phi_{\infty}[\textbf{e}^{n_{d}-1-n',\ldots,n_{1}-1,0}]
	\mod \Phi({}_R \mathcal{O}^{\sh}_{[0,d-2]}) .
	\end{multline}
\end{Proposition}

\begin{proof} (i) \label{lemma dg l3}
	By Lemma \ref{fact depth graded inversion}, $\tilde{l}_{\mu}$ preserves the depth filtration and we have $\gr_{D}\tilde{l}_{\mu}(w) = 0$.
	\noindent More precisely, if $w=\textbf{e}^{n_{d}-1,\ldots,n_{2}-1,0,0}$ with $n_{d} \geq 2$, we have : $\tilde{l}_{\mu}(w) \equiv - \mu \tilde{\partial}w \mod {}_R \mathcal{O}^{\sh}_{[0,d-2]}$,
	\newline\indent (ii) \label{lemma dg l4} $\tilde{l}^{\infty}_{\mu}$ preserves the depth filtration and we have :
	\newline $\gr_{D}\tilde{l}_{\mu}^{\infty}(w) =
	\sum\limits_{\substack{0 \leq n'_{0} \leq n_{0}-1 \\ 0 \leq n'_{d} \leq n_{d}-1}}
	\frac{(-1)^{n'_{0}}}{n'_{d}!n'_{0}!} \big( \frac{\mu}{2} \big)^{n'_{d}+n'_{0}} \bigg(
	\sum\limits_{0 \leq l \leq n_{d}-1-n'_{d}}  
	\frac{\mu^{l}}{l!} \partial^{l+n'_{d}}\tilde{\partial}^{n'_{0}}w- \sum\limits_{0\leq l \leq n_{0}-1-n'_{0}}
	\mu^{l} \partial^{n'_{d}}\tilde{\partial}^{n'_{0}+l}w \bigg)$. In particular, if $w=\textbf{e}^{n_{d}-1,\ldots,n_{1}-1,0}$ with $n_{d}\geq 2$, we have :
	\newline $\gr_{D}\tilde{l}_{\mu}^{\infty}(w) =
	\sum\limits_{0 \leq n'_{d} \leq n_{d}-1}
	\frac{1}{n'_{d}!} \big(\frac{\mu}{2}\big)^{n'_{d}}
	\bigg( \sum\limits_{0 \leq l \leq n_{d}-1-n'_{d}}\frac{\mu^{l}}{l!} \text{ } \partial^{l+n'_{d}} w - \partial^{n'_{d}}w \bigg)$
	\newline\indent (iii) Let us prove the equality of modules (\ref{eq:equality of modules}) by induction on $d$.
	\newline\indent We prove first the result for $d=1$. By the proof of Proposition \ref{main theorem 0}, (iii), we have $\Phi[e_{0}^{n}]=\Phi_{\infty}[e_{0}^{n}]=0$ for any positive integer $n$, and $\Phi[e_{1}]=\Phi_{\infty}[e_{1}]=0$. We specialize equation (\ref{eq:summary of equations avec tilde mu}) to $\textbf{e}^{n-1,0}$ with $n\geq 2$. We obtain : 
	$0=\sum\limits_{\substack{l_{1},l_{2}\geq 0 \\ l_{1}+l_{2}\leq n}} e^{\frac{\mu}{2} e_{0}}[\textbf{e}^{l_{1}}] e^{\mu e_{\infty}}[\textbf{e}^{l_{2}}]f_{\infty}[\textbf{e}^{n-l_{1}-l_{2},0}] = 
	\sum\limits_{\substack{l_{1},l_{2}\geq 0 \\ l_{1}+l_{2}\leq n}}
	\frac{\mu^{l_{1}+l_{2}}}{2^{l_{1}}l_{1}!l_{2}!}
	f_{\infty}[\textbf{e}^{n-l_{1}-l_{2},0}]$. Whence, by induction on $n$, $f_{\infty}[\textbf{e}^{n-1,0}]=0$. By \S2.1.2 this implies that $f_{\infty}(\mathcal{O}^{\sh}_{[0,1]})=0$.
	\newline\indent Let us now assume the result for $d-1$. \newline\indent By (i), equation (\ref{eq:summary of equations avec tilde mu}) specified to a word of the form $\textbf{e}^{n_{d}-1,\ldots,n_{1}-1,0,0}$ with $n_{d}\geq 2$ gives
	\begin{multline*} \Phi[\textbf{e}^{n_{d}-1,\ldots,n_{1}-1,0}] \in
	\\
	\sum_{0 \leq n'_{d} \leq n_{d}-1}
	\frac{1}{n'_{d}!} \bigg(\frac{\mu}{2}\bigg)^{n'_{d}} \bigg(
	\sum_{1 \leq l \leq n_{d}-1-n'_{d}}\frac{\mu^{l-1}}{l!} \text{ } \Phi_{\infty}[\textbf{e}^{n_{d}-1-l-n'_{d},\ldots,n_{1}-1,0,0}] - \Phi_{\infty}[\textbf{e}^{n_{d}-1-n'_{d},\ldots,n_{1}-1,0,0}] \bigg) 
	\\ + \Phi_{\infty}({}_R \mathcal{O}^{\sh}_{[0,d]} ) + \Phi({}_R \mathcal{O}^{\sh}_{[0,d-1]}); 
	\end{multline*}
	\noindent the induction hypothesis implies
	$\Phi({}_R \mathcal{O}^{\sh}_{[0,d-1]}) \subset \Phi_{\infty}({}_R \mathcal{O}^{\sh}_{[0,d]})$ ; by definition $\Phi_{\infty}({}_R \mathcal{O}^{\sh}_{[0,d]} ) \subset \Phi_{\infty}({}_R \mathcal{O}^{\sh}_{[0,d+1]})$, whence $\Phi[\textbf{e}^{n_{d}-1,\ldots,n_{1}-1,0}] \in \Phi_{\infty}({}_R \mathcal{O}^{\sh}_{[0,d+1]})$ ; by \S \ref{convergent words}, this last equality implies the inclusion $\Phi({}_R \mathcal{O}^{\sh}_{[0,d]}) \subset \Phi_{\infty}({}_R \mathcal{O}^{\sh}_{[0,d+1]})$ ; and this implies the congruence (\ref{eq:congruence 3 mu not 0}).
	\newline\indent By (ii), equation (\ref{eq:summary of equations avec tilde mu}) specified to a word of the form $\textbf{e}^{n_{d+1}-1,\ldots,n_{1}-1,0}$ with $n_{d+1} \geq 2$ gives
	\begin{multline*} \Phi_{\infty}\Big[\sum_{\substack{n'_{d+1} \geq 0,\text{ }l\geq 1 \\ 0 \leq n'_{d+1}+l \leq n_{d+1}-1}}
	\frac{\mu^{n'_{d+1}+l}}{n'_{d}!l!2^{n'_{d+1}}} \text{ } \textbf{e}^{n_{d+1}-1-n'_{d+1}-l,n_{d}-1,\ldots,n_{1}-1,0} \Big]
	\\ \in \Phi[\textbf{e}^{n_{d+1}-1,\ldots,n_{1}-1}] + \Phi_{\infty}({}_R \mathcal{O}^{\sh}_{[0,d]}) + \Phi({}_R \mathcal{O}^{\sh}_{[0,d-1]}) ;
	\end{multline*}
	\noindent the induction hypothesis implies $\Phi_{\infty}({}_R \mathcal{O}^{\sh}_{[0,d]}) \subset \Phi({}_R \mathcal{O}^{\sh}_{[0,d-1]})$ ; by definition, $\Phi({}_R \mathcal{O}^{\sh}_{[0,d-1]}) \subset \Phi({}_R \mathcal{O}^{\sh}_{[0,d]})$, whence $ \Phi_{\infty}\Big[\sum\limits_{\substack{n'_{d+1}\geq 0,\text{ }l\geq 1 \\ 0 \leq n'_{d+1}+l \leq n_{d+1}-1}}
	\frac{\mu^{n'_{d+1}+l}}{n'_{d+1}!l!2^{n'_{d+1}}} \text{ } e_{0}^{n_{d+1}-1-n'_{d+1}-l}e_{1}e_{0}^{n_{d}-1}e_{1}\cdots e_{0}^{n_{1}-1}e_{1}\Big] \in \Phi({}_R \mathcal{O}^{\sh}_{[0,d]})$ ; by \S\ref{convergent words}, and an induction on $n_{d+1}$, this implies the converse inclusion $\Phi_{\infty}({}_R \mathcal{O}^{\sh}_{[0,d+1]} ) \subset \Phi({}_R \mathcal{O}^{\sh}_{[0,d]})$ and the congruence (\ref{eq:congruence 4 mu not 0}).
	\newline\indent (iv) By Lemma \ref{fact depth graded inversion} and Lemma \ref{coeff automorphism}, $l_{\mu}$ preserves the depth filtration and we have $\gr_{D}l_{\mu}(w) = \sum\limits_{n'=0}^{n_{0}-1} (1 - (-1)^{\weight(w)-\depth(w)-n'}) \frac{1}{n'!}( \frac{\mu}{2})^{n'}  \tilde{\partial}^{n'}w$
	\noindent In particular, if $w=\textbf{e}^{n_{d}-1,\ldots,n_{1}-1,0}$ with $n_{d}\geq 2$, we have :
	$\gr_{D}l_{\mu}(w)= (1-(-1)^{\depth(w) - \weight(w)}) w$.
	\newline\indent (v) \label{lemma dg l2}$l^{\infty}_{\mu}$ preserves the depth filtration and we have $\gr_{D}l_{\mu}^{\infty}(w) =
	\sum\limits_{n'=0}^{n_{d}-1} \frac{1}{n'!} \big( \frac{\mu}{2} \big)^{n'} \partial^{n'}w$.
	\newline\indent (vi) By (iv), (v) and the equality of modules (\ref{eq:equality of modules}), the congruence (\ref{eq:congruence 12}) is the depth-graded version of equation (\ref{eq:summary of equations sans tilde}).
\end{proof}

\subsection{The adjoint depth reduction}

We deduce from Proposition \ref{main theorem} the depth reduction (i) of the Theorem for $M_{\mu}$.

\begin{Corollary} \label{depth red f adjoint} Let $\Phi \in \tilde{\Pi}(K)$ satisfying equations (\ref{eq:summary of equations avec tilde}) and (\ref{eq:summary of equations sans tilde}). For any positive integers $d$ and $n_{i}$ ($1 \leq i \leq d$), $(n_{1},\ldots,n_{d}) + (-1)^{n_{1}+\cdots+n_{d}}(n_{d},\ldots,n_{1})$ admits a depth reduction for $\Phi$ with coefficients in $R$.
\end{Corollary}

\begin{proof} We specialize the equation (\ref{eq:summary of equations avec tilde}) to the coefficients of the form $\textbf{e}^{0,n_{d}-1,\ldots,n_{1}-1,0}$. We obtain a formula for $(\Phi^{-1}e_{1}\Phi)[\textbf{e}^{0,n_{d}-1,\ldots,n_{1}-1,0}]$ in $\Phi_{\infty}(\mathcal{O}^{\sh}_{[0,d]})$. Applying the expression of $\Phi_{\infty}$ in depth $d$ in terms of $\Phi$ in depth $\leq d-1$ given by $\Phi \circ \tilde{l}_{\mu} = \Phi_{\infty} \circ \tilde{l}_{\mu}^{\infty}$ gives the result.
\end{proof}

\subsection{The parity depth reduction}

We deduce from Proposition \ref{main theorem} the depth reduction (ii) of the Theorem for $M_{\mu}$.

\begin{Corollary} \label{depth red f parity} Let $\Phi \in \tilde{\Pi}(K)$ satisfying equations (\ref{eq:summary of equations avec tilde}) and (\ref{eq:summary of equations sans tilde}). For any positive integers $d$ and $n_{i}$ ($1 \leq i \leq d$) such that $n_{1}+\cdots+n_{d}-d$ is odd,
	$(n_{1},\ldots,n_{d})$ admits a depth reduction for $\Phi$, with coefficients in $R$.
\end{Corollary}

\begin{proof} This follows from the equality of modules (\ref{eq:equality of modules}) and the congruence (\ref{eq:congruence 12}).
\end{proof}

\begin{Remark} The analogue of the parity depth reduction for $\Phi_{\infty}$ is that, if $n_{d}+\cdots+n_{1}-d$ is odd, $\sum\limits_{l=0}^{n_{d}-1} \frac{1}{l!} \big(\frac{\mu}{2} \big)^{l} (n_{1},\ldots,n_{d-1},n_{d}-l)$ admits a depth reduction for $\Phi_{\infty}$ with coefficients in $R$. This follows again from equations (\ref{eq:equality of modules}) and (\ref{eq:congruence 12}).
\end{Remark}

We can combine the adjoint depth reduction and the parity depth reduction.

\begin{Corollary} Let $\Phi \in \tilde{\Pi}(K)$ satisfying equations (\ref{eq:summary of equations avec tilde}) and (\ref{eq:summary of equations sans tilde}). For any positive integers $d$ and $n_{i}$ ($1 \leq i \leq d$) such that $(n_{1}+\cdots+n_{d})-d$ is even, $(n_{1},\ldots,n_{d}) + (-1)^{n_{1}+\cdots+n_{d}}(n_{d},\ldots,n_{1})$ has a depth reduction down to $d-2$ for $\Phi$ with coefficients in $R$.
\end{Corollary}

\begin{proof} Follows from Corollary \ref{depth red f adjoint} and Corollary \ref{depth red f parity}.
\end{proof}

The results of this section apply in particular when $K=\mathbb{C}$, $\Phi$ is the non-commutative generating series of multiple zeta values and $\mu=2\pi i$. The parity depth reduction for MZV's has been proved in \cite{Tsumura}, \S1. As mentioned in the introduction, it was first in depth 1 and 2 by Euler. The first example in depth two is the equality $\zeta(1,2)=\zeta(3)$. The parity depth reduction has been proved for multiple polylogarithms in \cite{Panzer}. For solutions to the double shuffle relations, the parity depth reduction was proved in \cite{IKZ}, \S8, and the adjoint depth reduction in \cite{Yasuda}, Proposition 3.3.

\section{Applications to the study of $p$-adic multiple zeta values via explicit formulas}

The results of \S3 and \S4 can be applied, respectively to  $\Phi_{p,\alpha} \in M_{0}(\mathbb{Q}_{p})$ and $\Phi_{\KZ} \in M_{2\pi i}(\mathbb{R})$, the non-commutative generating series of $p$-adic respectively real multiple zeta values. The results of \S4 can also the lift of $\Phi_{p,\alpha}$ with coefficients in $\B_{\dR}$, and $\mu$ equal to $t$, the $p$-adic analogue of $2\pi i$, which is a uniformizer and must be here considered to be of weight $1$. We now discuss some particular applications to the study of $p$-adic multiple zeta values via explicit formulas.

\subsection{The relation between ($p$-adic) multiple zeta values and their adjoint analogues}

The goal of \cite{J1, J2, J3} was to find explicit formulas for $p$-adic multiple zeta values, as sums of series, $p$-adic analogues of (\ref{eq:mzv}). What we found is such formulas but, not for $p$-adic multiple zeta values themselves, but for the following variant, which we called adjoint $p$-adic multiple zeta values in \cite{J4} :

\begin{equation} \label{eq:number adjoint} \begin{array}{ll}
\zeta_{p,\alpha}^{\Ad}(n_{1},\ldots,n_{d};l) & = \displaystyle  (\Phi_{p,\alpha}^{-1}e_{1}\Phi_{p,\alpha})[e_{0}^{l}e_{1}e_{0}^{n_{d}-1}e_{1} \cdots e_{0}^{n_{1}-1}e_{1}] 
\\ & = \displaystyle  \sum_{d'=0}^{d} \sum_{\substack{l_{d'+1},\ldots,l_{d}\geq 0 \\ l_{d'+1} + \cdots + l_{d} = l}} \prod_{i=d'+1}^{d} {-n_{i} \choose l_{i}}
\zeta_{p,\alpha}(n_{1},\ldots,n_{d'}) \zeta_{p,\alpha}(n_{d}+l_{d},\ldots,n_{d'+1}+l_{d'+1})
\end{array}
\end{equation}

The numbers (\ref{eq:number adjoint}) are central for us to deal with explicit formulas for $p$-adic multiple zeta values :fFor any question on $p$-adic multiple zeta values that we want to study via explicit formula, we must first find and solve the ``adjoint'' variant of the question which involves these numbers.

Then, in \cite{J4, J5, J6} we have related the motivic Galois theory of $p$-adic multiple zeta values and the formulas found in the previous papers. There, we consider that for any question on $p$-adic MZV's which one wants to study by explicit formulas, one has first to find its adjoint variants and to solve it. This strategy is the only one which arises naturally from our computations of \cite{J1, J2, J3}.

An index $(n_{1},\ldots,n_{d};l)$ being identified to $e_{0}^{l}e_{1}e_{0}^{n_{d}-1}e_{1}\ldots e_{0}^{n_{1}-1}e_{1}$, we say that it has depth $d+1$ and weight $n_{1}+\cdots+n_{d}+l+1$.

A small variation on the proof of proposition 3.2 gives the following result, which clarifies the relation between $p$-adic multiple zeta values and their adjoint variants, and justifies that in \cite{J1,J2,J3,J4,J5,J6} we can consider only adjoint $p$-adic multiple zeta values and not directly $p$-adic multiple zeta values.

\begin{Proposition} \label{proposition adjoint 1}
The following $\mathbb{Q}$-vector spaces are equal : 

(i) The vector space $V_{n,\leq d}$ generated by $p$-adic multiple zeta values $\zeta_{p,\alpha}(w)$ with $w$ of weight $n$ and depth $\leq d$.
	
(ii) The vector space $V^{\Ad}_{n+1,\leq d+1}$ generated by adjoint $p$-adic multiple zeta values $\zeta^{\Ad}_{p,\alpha}(w)$ with $w$ of weight $n$ and depth $\leq d+1$.
\end{Proposition}

\begin{proof} The inclusion $V^{\Ad}_{n+1,\leq d+1} \subset V_{n,\leq d}$ follows from the shuffle relation (\S1.1.4 a). Let us prove $V_{n,\leq d} \subset V^{\Ad}_{n+1,\leq  d+1}$ by induction on $d$. For $d=0$, we have $V_{n,0} = \{0\}$ and the result is clear. Assume the result true for $d-1$.

Let a word $w$ on $\{e_{0},e_{1}\}$ of depth $d-1$ and weight $n-2$, thus $e_{0}we_{1}$ has depth $d$ and weight $n$. By the fact 2.1 (ii), we have $(\Phi^{-1}e_{1}\Phi)[e_{0}w e_{1}^{2}] \in \Phi^{-1}[e_{0}w e_{1}] + \Phi(\mathcal{O}^{\sh}_{\ast,\leq d-1})$, whence $\Phi^{-1}[e_{0}w e_{1}] \in (\Phi^{-1}e_{1}\Phi)[e_{0}w e_{1}^{2}] + \Phi(\mathcal{O}^{\sh}_{\ast,\leq d-1}) \subset V^{\Ad}_{n+1,d+1} + V_{n,\leq d-1} \subset V^{\Ad}_{n+1,\leq d+1}$, where the last inclusion follows from the induction hypothesis.

On the other hand, by the fact 2.1 (ii), we have $\Phi^{-1}[e_{0}w e_{1}] \in -\Phi[e_{0}w e_{1}] + \Phi(\mathcal{O}^{\sh}_{\ast,\leq d-1}) = -\Phi[e_{0}w e_{1}] + V_{n,\leq d-1} \in V^{\Ad}_{n+1,\leq d+1}$ by the induction hypothesis.

We deduce that $\Phi[e_{0}w e_{1}] \in V^{\Ad}_{n+1,\leq d+1}$. Since this is true for all $w$, this proves the result for $d$.
\end{proof}

\begin{Remark} In \cite{J4} we also introduce and study complex adjoint multiple zeta values as follows : 

$$ \zeta^{\Ad}(n_{1},\ldots,n_{d};l) = (\Phi_{\KZ}^{-1}e^{2\pi i e_{1}} \Phi_{\KZ})[e_{0}^{l} e_{1} e_{0}^{n_{d}-1}e_{1} \cdots e_{0}^{n_{1}-1}e_{1}] $$

One has a complex analogue of the proposition \ref{proposition adjoint 1} which is proved similarly.
\end{Remark}

\begin{Remark} Let $N$ be a positive integer. The de Rham pro-unipotent fundamental groupoid of $(\mathbb{P}^{1} - \{0,\mu_{N},\infty\})/K$, where $K$ is a field of characteristic $0$ which contains a primitive $N$-th root of unity, admits a description which can be obtained from the one of $\pi_{1}^{\un,\dR}(\mathbb{P}^{1} - \{0,1,\infty\})$ of \S1.1, by replacing the alphabet $\{e_{0},e_{1}\}$ by the alphabet $\{e_{0} \} \cup \{e_{\xi}\text{ }|\text{ }\xi^{N} = 1 \}$.
Multiple zeta values have cyclotomic generalizations, as well as their $p$-adic analogues if $p$ does not divide $N$. Their indices are words 
$\big( (n_{i})_{d},(\xi_{i})_{d}\big) = e_{0}^{n_{d}-1}e_{\xi_{d}} \ldots e_{0}^{n_{1}-1}e_{\xi_{1}}$. $p$-adic cyclotomic multiple zeta values are studied in \cite{J1, J2, J3, J4, J5, J6}. The proposition \ref{proposition adjoint 1} can be easily generalized to this case.
\end{Remark}

\subsection{A new point of view on the relation between $p$-adic and finite multiple zeta values}

\subsubsection{The integrality of $p$-adic multiple zeta values}

The pro-unipotent fundamental groupoid has a cohomological interpretation, by a theorem of Beilinson, reviewed in \cite{Goncharov} \S4, and developed in \cite{Deligne Goncharov} \S3, in order to construct the motivic pro-unipotent fundamental groupoid. As a consequence, the crystalline realization of the pro-unipotent fundamental groupoid can be related to log-crystalline cohomology. Using this cohomological interpretation and the comparison between Frobenius and the Hodge filtration on log-crystalline cohomology, Akagi, Hirose and Yasuda have proved the following \cite{AHY}:

\begin{equation} \label{eq:integrality} \zeta_{p}^{\KZ}(n_{1},\ldots,n_{d}) \in \sum_{n> n_{1}+\ldots+n_{d}} \frac{p^{n}}{n!}\mathbb{Z}_{p} 
\end{equation}

This has also been proved independently by Chatzistamatiou \cite{Cha}. We note that in the equations of \S3, the rational coefficients are in $\mathbb{Z}$ and thus, if we specialize these equations to $p$-adic multiple zeta values, each term of weight $N$ is in $\sum\limits_{n>N} \frac{p^{n}}{n!}\mathbb{Z}_{p}$.

The integrality property can be lifted to $B_{dR}$ : the lift of the above $p$-adic multiple zeta value is in $\sum\limits_{n> n_{1}+\ldots+n_{d}} \frac{t^{n}}{n!}B_{dR}^{+}$ where $t$ is the uniformizer of $B_{dR}^{+}$. If $p\not=2$, in the equations of \S4 specified to this lift of $p$-adic multiple zeta values in $B_{dR}$, with $\mu=t$ which has valuation 1 and weight $1$, each term of weight $n$ is in $\sum\limits_{n> n_{1}+\ldots+n_{d}} \frac{t^{n}}{n!}B_{dR}^{+}$.

It is interesting to study the connection between this integrality property and our theory which uses explicit formulas. Our explicit computation does not use log-crystalline cohomology, but a sort of equivalent of rigid cohomology, on which there is no integral strcture, unlike on log-crystalline cohomology. A priori there is no reason to hope for recovering the integrality by the explicit formulas. However, an interesting phenomenon occurs, that the formulas reflect at least partially the integrality. This phenomenon is expressed in an original way by the notion of finite multiple zeta values.

\subsubsection{Finite multiple zeta values}

Kaneko and Zagier conjecture that the following formula defines an isomorphism between the $\mathbb{Q}$-algebra generated by finite MZV's and the $\mathbb{Q}$-algebra of generated by MZV's moded out by the ideal $(\zeta(2))$ : 
\begin{equation} \label{eq:conjecture Kaneko Zagier} \zeta_{A}(n_{1},\ldots,n_{d}) \mapsto \sum_{d'=0}^{d} (-1)^{n_{d'+1}+\cdots+n_{d}} \zeta(n_{1},\ldots,n_{d'}) \zeta(n_{d},\ldots,n_{d'+1}) \mod \zeta(2) .
\end{equation}

Let us note that the right-hand side in (\ref{eq:conjecture Kaneko Zagier}) is also the following numbers, which appear in our result of adjoint depth reduction
\begin{equation} \label{eq:adjoint0} (\Phi_{\KZ}^{-1}e_{1}\Phi_{\KZ})[e_{1}e_{0}^{n_{d}-1}e_{1} \ldots e_{0}^{n_{1}-1}e_{1}] \mod \zeta(2) .
\end{equation}

We have proved in \cite{J2} the following result, which was conjectured by Akagi, Hirose and Yasuda for $\alpha=1$ :

\begin{equation} \label{eq: equation of J2} (p^{\alpha})^{n_{1}+\cdots+n_{d}}\sum_{0<m_{1}<\cdots<m_{d}<p^{\alpha}} \frac{1}{m_{1}^{n_{1}} \cdots m_{d}^{n_{d}}} = \sum_{l=0}^{\infty} \zeta_{p,\alpha}^{\Ad}(n_{1},\ldots,n_{d};l) 
\end{equation}

Let the map of ``reduction modulo infinitely large primes'':
$$ \red_{/ p \rightarrow \infty} :  (x_{p})_{p} \in \prod'_{p\in\mathcal{P}}\mathbb{Q}_{p}= \{(x_{p}) \in \prod_{p} \mathbb{Q}_{p}\text{ }|\text{ } \text{for p large}, x_{p} \in \mathbb{Z}_{p}\} \rightarrow (x_{p} \mod p)_{p\text{ large}} \in \mathcal{A}$$ 

The integrality (\ref{eq:integrality}) implies that $\zeta_{p}^{\KZ}(n_{1},\ldots,n_{d}) \in p^{n_{1}+\cdots+n_{d}} \mathbb{Z}_{p}$ for $p>n_{1}+\ldots+n_{d}+1$. Denoting by $\zeta_{p}^{\De} = p^{-\weight} \zeta_{p,1}$, by \cite{Furusho MZV2} theorem 2.8, this implies $\zeta_{p}^{\KZ}(n_{1},\ldots,n_{d}) \in \mathbb{Z}_{p}$ for $p>n_{1}+\ldots+n_{d}+1$.

Akagi, Hirose and Yasuda have joined equations (\ref{eq:integrality}) and (\ref{eq: equation of J2}) to deduce :

\begin{equation} \label{eq:finite in terms of p-adic}
\zeta_{\mathcal{A}}(n_{1},\ldots,n_{d}) = \red_{/ p \rightarrow \infty} (\zeta_{p}^{\De}(n_{1},\ldots,n_{d}))_{p\in\mathcal{P}}
\end{equation}

Moreover, Yasuda has proved, as the main result of \cite{Yasuda}, that the numbers (\ref{eq:adjoint0}) generate the $\mathbb{Q}$-vector space of MZV's mod $\zeta(2)$, and similarly for $p$-adic mutliple zeta values (his proof works for all solutions to the double shuffle equations). This combined to (\ref{eq:finite in terms of p-adic}) implies that the image of the $\mathbb{Q}$-algebra generated by the numbers $(\zeta_{p}^{\De}(n_{1},\ldots,n_{d}))_{p\in\mathcal{P}}$ by $\red_{\red_{/ p \rightarrow \infty}}$ is exactly the $\mathbb{Q}$-algebra generated by finite multiple zeta values. It also gives a way to write the image of $(\zeta_{p}^{\De}(n_{1},\ldots,n_{d}))_{p\in\mathcal{P}}$ in terms of finite multiple zeta values.

\begin{Example} (i) If $p-1$ divides $n$, then $\sum\limits_{0<m<p}\frac{1}{m^{n}} \equiv -1 \mod p$ and otherwise $\sum\limits_{0<m<p}\frac{1}{m^{n}} \equiv 0 \mod p$.
In particular, for any $n$ we have $\zeta_{\mathcal{A}}(n) = 0$.
\newline (ii) (\cite{Hoffman}, theorem 6.1)
If $p> n_{1}+n_{2}$, then
$\sum\limits_{0<m_{1}<m_{2}<p} \frac{1}{m_{1}^{n_{1}}m_{2}^{n_{2}}} \equiv (-1)^{n_{2}-1}{n_{1}+n_{2} \choose n_{1}} \frac{B_{p-n_{1}-n_{2}}}{n_{1}+n_{2}} \mod p \equiv \frac{-1}{n_{1}+n_{2}} {n_{1}+n_{2} \choose n_{1}}\sum\limits_{0<m_{1}<m_{2}<p} \frac{1}{m_{1}^{n_{1}+n_{2}-1}m_{2}} \mod p$.
In particular, for any $n_{1},n_{2}$, we have
$\zeta_{\mathcal{A}}(n_{1},n_{2}) = $ \newline $(-1)^{n_{2}-1}{n_{1}+n_{2} \choose n_{1}}\zeta_{\mathcal{A}}(1,n_{1}+n_{2}-1)$, and  $\zeta_{\mathcal{A}}(n_{1},n_{2})=0$ if $n_{1}+n_{2}$ is even.
\newline For $p>n$, by the above discussion and (\ref{eq:finite 1}) the $p$-adic zeta value $\zeta_{p,1}(n)=-\Phi_{p,1}[\textbf{e}^{n-1,0}]$ is in $p^{n}\mathbb{Z}_{p}$ and is congruent to $p^{n}\frac{B_{p-n}}{p-n}$ modulo $p^{n+1}$. With our notation, for any word $w$, $\zeta_{p,1}(w)$ means Deligne's $p$-adic multiple zeta value $\zeta_{p}(w)$ defined in \cite{Deligne Goncharov}, \S5.28 multiplied by $p^{\weight(w)}$.
\end{Example}

\subsubsection{A motivic lift of the integrality}

In our explicit version of the algebraic theory of $p$-adic multiple zeta values, \cite{J4,J5,J6}, the main objects are not $p$-adic and finite multiple zeta values, but the two following objets defined in \cite{J4} : adjoint $p$-adic multiple zeta values (\ref{eq:number adjoint}), and \emph{multiple harmonic values} :

$$ \har_{\mathcal{P}^{\mathbb{N}}}(n_{1},\ldots,n_{d}) = \Big( (p^{\alpha})^{n_{1}+\cdots+n_{d}}\sum_{0<m_{1}<\cdots<m_{d}<p^{\alpha}} \frac{1}{m_{1}^{n_{1}} \cdots m_{d}^{n_{d}}} \Big)_{(p,\alpha) \in \mathcal{P} \times \mathbb{N}_{\geq 1}} \in \bigg( \prod_{p \in \mathcal{P}} \mathbb{Q}_{p} \bigg)^{\mathbb{N}} $$

We show in \cite{J4, J5, J6} that the main properties of finite multiple zeta values have an adelic lift to properties of multiple harmonic values, and that the main properties of the numbers (\ref{eq:adjoint0}) have generalizations to properties of adjoint multiple zeta values.

Our theory is expressed as a relation between properties of adjoint $p$-adic multiple zeta values and properties multiple harmonic values. We define in \cite{J6} the motivic multiple harmonic values as follows. They are elements of the weight-adic completion $\mathcal{Z}^{\mot} = \prod_{n=0}^{\infty} \mathcal{Z}_{n}^{\mot}$of the weight-graded algebra $\mathcal{Z}^{\mot} = \oplus_{n=0}^{\infty} \mathcal{Z}_{n}^{mot}$ of motivic multiple zeta values, $\mathcal{Z}_{n}^{\mot}$ being the $\mathbb{Q}$-vector space of motivic multiple zeta values of weight $n$. The definition is motivated by equation (\ref{eq: equation of J2}) and results of \cite{J4, J5, J6}:

$$ \har^{\mot}(n_{1},\ldots,n_{d}) = \sum_{l=0}^{\infty}
(\Phi^{-1}e_{1}\Phi)^{\mot}[e_{0}^{l}e_{1}e_{0}^{n_{d}-1}e_{1} \ldots e_{0}^{n_{1}-1}e_{1}] $$

The goal of \cite{J6} is to study how multiple harmonic values can be regarded as ``periods'' in a certain sense, although they are not technically periods. This interprets phenomena observed in \cite{J1, J2, J3, J4, J5}. One useful result is the following, a completed motivic lift of equation (\ref{eq:finite in terms of p-adic}) :

\begin{Proposition} For each $w=e_{1}e_{0}^{n_{d}-1}e_{1}\ldots e_{0}^{n_{1}-1}e_{1}$, there exists a sequence $w_{n}$ such that $w_{n}$ is a $\mathbb{Q}$-linear combination of words of weight equal to $\weight(w) + n$, such that $(\Phi^{-1}e_{1}\Phi)^{\mot}[w] = \sum_{n=0}^{\infty} \har^{\mot}(w_{n})$. 

In particular, the weight-adic completion of the algebra of motivic multiple zeta values is topologically generated by the motivic multiple harmonic values.
\end{Proposition}

\begin{proof} Note that Yasuda's main theorem \cite{Yasuda} applies to motivic multiple zeta values because they satisfy the double shuffle relations.
	
We prove by induction on n the existence of $w_{i}$, $1 \leq i \leq n$ and of $z_{i}$, $i\geq n$, such that 

$(\Phi^{-1}e_{1}\Phi)^{\mot}[w] = \sum\limits_{i=0}^{n} \har^{\mot}(w_{i}) + \sum\limits_{l=n+1}^{\infty} (\Phi^{-1}e_{1}\Phi)[e_{0}^{l}e_{1}w]$.

For $n=0$ we write $(\Phi^{-1}e_{1}\Phi)^{\mot}[w] = \har^{\mot}(w) - (\Phi^{-1}e_{1}\Phi)[e_{0}^{n}e_{1}w]$.

Assuming the result is true for $n$, by the main theorem of \cite{Yasuda}, there is a $\mathbb{Q}$-linear combination $w_{n+1}=e_{1}u_{n}$ of words of the form $e_{1}\cdots e_{1}$ such that we have $(\Phi^{-1}e_{1}\Phi)^{\mot}[e_{0}^{n+1}e_{1}w] = (\Phi^{-1}e_{1}\Phi)^{\mot}(e_{1}u_{n})$. Moreover, we have 
$(\Phi^{-1}e_{1}\Phi)^{\mot}(e_{1}u_{n}) = \har^{\mot}(u_{n}) - \sum\limits_{l=1}^{\infty} (\Phi^{-1}e_{1}\Phi)(e_{0}e_{1}u_{n})$. Whence the result for $n+1$.

Thus the result is true for all $n$, and we deduce the proposition by taking the limit $n \rightarrow \infty$.
\end{proof}

In \cite{J6} we study the analogue of the period conjecture for multiple harmonic values. This question is related to the difference between the rigid cohomology and the log-crystalline cohomology which represent $\pi_{1}^{\un}(\mathbb{P}^{1} - \{0,1,\infty\},\vec{1}_{1},\vec{1}_{0})$. We also relate it in \cite{J6} to the question of which part of the information on the valuation of $p$-adic multiple zeta values and multiple harmonic sums is conjecturally of motivic nature.

It is required for that purpose to bound the $p$-adic norm of the coefficients in this relation. This enables to see in which quotient of $\prod_{p\in\mathcal{P}}\mathbb{Z}_{p}$ the crystalline realization of this motivic equation can converge.

Our theorem of depth reduction brings the following contribution to the study of this question. The proof of the fact that the numbers $(\Phi^{-1}e_{1}\Phi)^{\mot}[e_{1}e_{0}^{n_{d}-1}e_{1} \ldots e_{0}^{n_{1}-1}e_{1}]$ generate the space of motivic multiple zeta values provided by \cite{Yasuda}, which we used in the proof of proposition 5.5, has two steps.

The first and main step is a proof of Corollary 3.4 tensorized with $\mathbb{Q}$, using not associator equations but double shuffle equations. Our proof using associators is simpler, and it also gives a bound on the $p$-adic norm of the denominators, which is similar to the bounds on $p$-adic norms of rational coefficients which appear in \cite{J1, J2, J3, J4, J5, J6}. This type of bounds ensure a uniform convergence in $\prod_{p\in\mathcal{P}}\mathbb{Z}_{p}$. Thus our theorem of depth reduction is a step in the study of the relation between the integrality of $p$-adic multiple zeta values and their explicit formulas.

The second step is a proof that the numbers in the right-hand side of equation (3.9) generate the space of depth-graded multiple zeta values.

\subsection{An algebraic property behind the pole of the Kubota-Leopoldt $p$-adic zeta function\label{paragraph interpretation pole}}

Let $L_{p}(s,\chi)$ be the $p$-adic zeta function of Kubota and Leopoldt, which is a $p$-adic analogue of the Riemann zeta function. Its special values coincide with $p$-adic MZV's of depth one, as follows : by \cite{Co}, equation (4) p. 173), we have, for all $n \geq 2$, denoting Teichm\"{u}ller's character by $\omega$ :
$$ \zeta_{p,1}(n) = p^{n} L_{p}(n,\omega^{1-n}). $$

This raises the question of finding interpolation of $p$-adic MZV's. We study this question in \cite{J7}.

As a preliminary to this study, here is a property of $p$-adic multiple zeta values which generalizes the classical fact that $L_{p}$ has a simple pole at $s=1$, and which is the consequence of our theorem of depth reduction for associators.

We start with the counterpart for $\Phi_{\infty}$ of the adjoint depth reduction (Corollary \ref{depth red f adjoint 0}) :

\begin{Lemma} \label{depth red finfty adjoint} Let $\Phi \in \tilde{\Pi}(K)$ which satisfies equations (\ref{eq:CH0 mod}) and (\ref{eq:hexagon part mu equal 0}). For any positive integers $d$ and $n_{i}$ ($0 \leq i \leq d$) and any non-negative integers $r_{d},r_{0}$,
	\begin{multline}
	{n_{0}+n_{d}+r_{0}+r_{d}-1 \choose n_{0}+n_{d}-1} 
	\Phi_{\infty}[\textbf{e}^{n_{d}-1+r_{d};n_{d-1}-1,\ldots,n_{1}-1;n_{0}-1+r_{0}}]
	\\ -\sum_{\substack{u_{1},\ldots,u_{d-1} \\ u_{d-1}+\cdots+u_{1}=r_{d}+r_{0}}}
	\prod_{d'=1}^{d-1} { -n_{d'} \choose u_{d'} } 
	\Phi_{\infty}[\textbf{e}^{n_{d}-1;n_{d-1}-1+u_{d-1},\ldots,n_{1}-1+u_{1};n_{0}-1}] 
	\end{multline}
	admits a depth reduction for $\Phi_{\infty}$.
\end{Lemma}

\begin{proof} Let
	$\widetilde{\Ad}(\Phi)(e_{1}) = \Phi^{-1}e_{1}\Phi - e_{1}$. Equation (\ref{eq:CH0 mod}) can be rewritten as
	\begin{equation} \label{eq:modifie}
	e_{0}\Phi_{\infty} - \Phi_{\infty}e_{0} = -(e_{1}\Phi_{\infty} - \Phi_{\infty}e_{1}) + \Phi_{\infty}\widetilde{\Ad}(\Phi)(e_{1}) .
	\end{equation}
	\noindent We note that, since $\Phi[e_{0}] = 0$, $\widetilde{\Ad}(\Phi)(e_{1})$ vanishes in depth $0$ and $1$. This and the equation (\ref{eq:equality of modules 0}) imply that we have
	\begin{equation} \label{eq:mod products}
	\Phi_{\infty}\widetilde{\Ad}(\Phi)(e_{1}) (\mathcal{O}^{\sh}_{[d]}) \subset \Phi_{\infty}(\mathcal{O}^{\sh}_{[0,d - 2]}) .
	\end{equation}
	\noindent Let a positive integer $d$ and non-negative integers $t_{i}$ $(1 \leq i \leq d)$. 
	By considering the coefficient of $\textbf{e}^{t_{d}+1,t_{d-1},\ldots,t_{1},t_{0}+1}$ in (\ref{eq:modifie}), we obtain
	\begin{equation}\label{eq:A1} 
	\Phi_{\infty}[\textbf{e}^{t_{d},t_{d-1},\ldots,t_{1},t_{0}+1}] - 
	\Phi_{\infty}[\textbf{e}^{t_{d}+1,t_{d-1},\ldots,t_{1},t_{0}}] = (\Phi_{\infty}\widetilde{\Ad}(\Phi)(e_{1}))[\textbf{e}^{t_{d}+1,t_{d-1},\ldots,t_{1},t_{0}+1}] .
	\end{equation}
	\noindent On the other hand, the shuffle relation $\Phi_{\infty}[e_{0}\text{ }\sh\text{ } \textbf{e}^{t_{d},t_{d-1},\ldots,t_{1},t_{0}}]=0$ gives 
	\begin{multline}\label{eq:A2}
	(t_{d}+1)\Phi_{\infty}[\textbf{e}^{t_{d}+1,t_{d-1},\ldots,t_{1},t_{0}}] + (t_{0}+1)\Phi_{\infty}[\textbf{e}^{t_{d},t_{d-1},\ldots,t_{1},t_{0}+1}]
	= -\sum_{d'=1}^{d-1} (t_{d'}+1)\Phi_{\infty}[\textbf{e}^{t_{d},t_{d-1},\ldots,t_{d'}+1,\ldots,t_{1},t_{0}}] .
	\end{multline}
	\noindent The linear system formed by (\ref{eq:A1}) and (\ref{eq:A2}) can be inverted as follows, after the change of notation which replaces $t_{0}$ by $t_{0}-1$, resp. $t_{d}$ by $t_{d}-1$ :
	\begin{multline} \label{eq:system1} 
	\Phi_{\infty}[\textbf{e}^{t_{d},\ldots,t_{1},t_{0}}] 
	= \sum_{d'=1}^{d-1} \frac{-(t_{d'}+1)}{t_{d}+t_{0}+1} \Phi_{\infty}[\textbf{e}^{t_{d},\ldots,t_{d'}+1,\ldots,t_{0}-1}]
	+ \frac{t_{d}+1}{t_{d}+t_{0}+1} (\Phi_{\infty}\widetilde{\Ad}(\Phi)(e_{1}))[\textbf{e}^{t_{d}+1,t_{d-1},\ldots,t_{1},t_{0}-1}] ,
	\end{multline}
	\begin{multline} \label{eq:system2}
	\Phi_{\infty}[\textbf{e}^{t_{d},t_{d-1},\ldots,t_{0}}] = \sum_{d'=1}^{d-1} \frac{-(t_{d'}+1)}{t_{d}+t_{0}+1} \Phi_{\infty}[\textbf{e}^{t_{d}-1,\ldots,t_{d'}+1,\ldots,t_{0}}]
	- \frac{t_{0}+1}{t_{d}+t_{0}+1} (\Phi_{\infty}\widetilde{\Ad}(\Phi)(e_{1}))[\textbf{e}^{t_{d}-1,t_{d-1},\ldots,t_{1},t_{0}+1}] .
	\end{multline}
	\noindent Let $(a_{d,i},a_{0,i})_{0\leq i \leq r_{d}+r_{0}}$ be any sequence of elements of $\{0,\ldots,r_{d}\} \times \{0,\ldots,r_{0}\}$, satisfying : 
	$$ \left\{ \begin{array}{l} (a_{d,0},a_{0,0})=(r_{d},r_{0})
	\text{ }\text{ }\text{ }\text{ and }\text{ }\text{ }\text{ }
	(a_{d,r_{d}+r_{0}},a_{0,r_{d}+r_{0}})=(0,0)
	\\ \forall \text{i }\in \{0,\ldots,r_{d}+r_{0}-1\},\text{ }\text{ } (a_{d,i+1},a_{0,i+1}) \in \{ (a_{d,i}-1,a_{0,i}),(a_{d,i},a_{0,i}-1) \} 
	\end{array} \right. $$
	\noindent The proof follows by induction on $(r_{d},r_{0})$ for the lexicographical order, using the the linear system formed by (\ref{eq:system1}), (\ref{eq:system2}), and equation (\ref{eq:mod products}) : we apply (\ref{eq:system1}), resp. (\ref{eq:system2}) inductively to $(t_{d},\ldots,t_{0}) = (n_{d}-1+a_{d,i},n_{d-1}-1,\ldots,n_{1}-1,n_{0}-1+a_{d,i})$ if $a_{0,i+1} = a_{0,i} - 1$, resp. $a_{d,i+1} = a_{d,i}-1$.
\end{proof}

We deduce a variant of equation (\ref{eq:congruence 3 mu = 0}) :

\begin{Proposition} \label{prop Kubota} Let $\Phi \in \tilde{\Pi}(K)$ which satisfies equations (\ref{eq:hexagon part mu equal 0}) and (\ref{eq:CH0 mod}). For any positive integers $d$ and $n_{i}$ ($1 \leq i \leq d$) such that $n_{d} \geq 2$, we have
\begin{multline} \label{eq:polar in all depths} \Phi[e_{0}^{n_{d}-1,\ldots,n_{1}-1,0}]
+ \frac{1}{n_{d}-1} \sum_{\substack{u_{1},\ldots,u_{d}\geq 0 \\ u_{d}+\cdots+u_{1}=n_{d}-2}}
\prod_{d'=1}^{d} {-n_{d'} \choose u_{d'}}
\Phi_{\infty}[\textbf{e}^{n_{d}-2;n_{d-1}-1+u_{d},\ldots,n_{1}-1+u_{2},u_{1},0}] \in \Phi(\mathcal{O}^{\sh}_{[0,d-1]}) .
\end{multline}
\end{Proposition}

\begin{proof} Lemma \ref{depth red finfty adjoint}, applied to $n_{d}=n_{0}=1$, $r_{0}=0$ implies, after renaming $r_{d}$ as $n_{d}-1$:
\begin{multline}
\Phi_{\infty}[\textbf{e}^{n_{d}-2;n_{d-1}-1,\ldots,n_{1}-1}e_{1}]
\\ - \frac{1}{n_{d}-1} \sum_{\substack{u_{1},\ldots,u_{d-1}\geq 0 \\ u_{d-1}+\cdots+u_{1}=n_{d}-2}}
	\prod_{d'=1}^{d-1} {-n_{d'} \choose u_{d'}}
	\Phi_{\infty}[\textbf{e}^{n_{d}-2;n_{d-1}-1+u_{d-1},\ldots,n_{1}-1+u_{1};n_{0}-1}] \in \text{DRed}_{d\rightarrow d-1}(\Phi_{\infty}) .
	\end{multline}
	\noindent On the other hand, by equation (\ref{eq:summary of equations avec tilde}) and equation (\ref{eq:equality of modules 0}), we have 
	$$ \Phi[\textbf{e}^{n_{d}-1,\ldots,n_{1}-1,0}]  + \Phi(\mathcal{O}^{\sh}_{[0,d-1]}) = -\Phi_{\infty}[\textbf{e}^{n_{d}-2,\ldots,n_{1}-1,0,0}] + \Phi_{\infty}(\mathcal{O}^{\sh}_{[0,d]}) . $$
	Whence the result.
\end{proof}

\begin{Example} (i) For any integer $n_{1} \geq 2$,
$$ \Phi[\textbf{e}^{n_{1}-1,0}] = \frac{(-1)^{n_{1}-1}}{n_{1}-1} \Phi_{\infty}[\textbf{e}^{0,n_{1}-1,0}] . $$
\noindent (ii) For any positive integers $n_{1},n_{2}$ such that $n_{2} \geq 2$ and $n_{2}+n_{1}$ is odd, we have :
$$ \Phi[\textbf{e}^{n_{2}-1,n_{1}-1,0}] = 
\frac{(-1)^{n_{2}-1}}{n_{2}-1} \sum_{l=0}^{n_{2}-2} {n-1+l \choose l} \Phi_{\infty}[\textbf{e}^{0,n_{1}-1+l,n_{2}-2+l,0}] + \Phi_{\infty}[\textbf{e}^{n_{2}-1,n_{1}-1,0}] . $$
In the case where $\Phi$ is the non-commutative generating series $\Phi_{p,-1}$ of $p$MZV's (here the subscript $-1$ means that the crystalline Frobenius of $\pi_{1}^{\un}(\mathbb{P}^{1} \setminus \{0,1,\infty\})$ is replaced by its inverse), (i) and (ii) above are written in \cite{Unver MZV}, respectively, \S5.11 and \S5.14 ; in (i) one can see the polar factor of the Kubota-Leopoldt $p$-adic $L$-function, and the regular factor written in terms of $\Phi_{\infty}$.
\end{Example}

Applying Proposition \ref{depth red finfty adjoint} with $n_{d} = n_{0} = 1$, to the right-hand side in Corollary \ref{prop Kubota}, we can make the formula of Corollary \ref{prop Kubota} more canonical : all coefficients $\Phi[e_{0}^{n_{d}-1}\cdots e_{1}]$ with $n_{d}\geq 2$ can be written as linear combinations of coefficients of the form $\Phi_{\infty}[e_{1} \cdots e_{1}]$ with some rational coefficients having ``poles'' attained for certain tuples $(n_{1},\ldots,n_{d})$ satisfying $n_{d}=1$ and other conditions, and coefficients of $\Phi$ at words of lower depth. This may be useful for studying $p$-adic multiple zeta values as functions of $(n_{1},\ldots,n_{d})$ viewed as a tuple of $p$-adic integers.


\begin{thebibliography}{50}
\bibitem[A]{Andre} Y. Andr\'{e}, \emph{Une introduction aux motifs (motifs purs, motifs mixtes, p\'{e}riodes)}, Panoramas et synth\`{e}ses, n°17, 2004, Soci\'{e}t\'{e} math\'{e}matique de France
\bibitem[AET]{AET} A. Alekseev, B. Enriquez, C. Torossian - \emph{Drinfeld associators, braid groups and explicit solutions of the Kashiwara-Vergne equations}, Publ. Math. Inst. Hautes Études Sci. 112 (2010), 143-189
\bibitem[AHY]{AHY} K. Akagi, M. Hirose, S. Yasuda - \emph{Integrality of $p$-adic multiple zeta values and application to finite multiple zeta values}, preprint
\bibitem[Cha]{Cha} A. Chatzistamatiou - \emph{On integrality of $p$-adic iterated integrals} - J. Algebra 474 (2017), 240-270
\bibitem[Co]{Co} R. Coleman - \emph{Dilogarithms, regulators and $p$-adic $L$-functions} - Invent. Math., 69 (1982), 2, 171-208
\bibitem[D]{Deligne} P. Deligne, \emph{Le groupe fondamental de la droite projective moins trois points}, Galois Groups over $\mathbb{Q}$ (Berkeley, CA, 1987), Math. Sci. Res. Inst. Publ. 16, Springer-Verlag, New York, 1989
\bibitem[DG]{Deligne Goncharov} P. Deligne, A. B. Goncharov, \emph{Groupes fondamentaux motiviques de Tate mixtes}, Ann. Sci. Ecole Norm. Sup. 38 (2005), 1, 1-56
\bibitem[Dr]{Drinfeld} V. G. Drinfeld - \emph{On quasitriangular quasi-Hopf algebras and on a group that is closely connected with $\Gal(\bar{\mathbb{Q}}/\mathbb{Q})$}", Algebra i Analiz, 2:4 (1990), 149-181 
\bibitem[E]{Enriquez} B. Enriquez - \emph{Quasi-reflection algebras and cyclotomic associators}, Selecta Math. (N.S.), 13 (2007) 3, 391-463
\bibitem[F1]{Furusho MZV1} H. Furusho - \emph{p-adic multiple zeta values I -- p-adic multiple polylogarithms and the p-adic KZ equation}, Invent. Math., 155 (2004), 2, 253-286
\bibitem[F2]{Furusho MZV2} H. Furusho - \emph{p-adic multiple zeta values II -- tannakian interpretations}, Amer. J. Math, 129, (2007), 4, 1105-1144
\bibitem[F3]{Furusho assoc ds} H. Furusho - \emph{Pentagon and hexagon equations}, Ann. of Math., 171 (2010), 1, 545-556
\bibitem[F4]{Furusho assoc} H. Furusho - \emph{Double shuffle relation for associators}, Ann. of Math., 174 (2011), 1, 341-360

\bibitem[G]{Goncharov} A. B. Goncharov - \emph{Multiple polylogarithms and mixed Tate motives}, arXiv:0103059v4

\bibitem[H]{Hoffman} M. Hoffman - \emph{Quasi-symmetric functions and mod p multiple harmonic sums}, Kyushu J. math., 69 (2004) 2
\bibitem[IKZ]{IKZ} K. Ihara, M. Kaneko, D. Zagier - \emph{Derivation and double shuffle relations for multiple zeta values}, Compos. Math. 142 (2006) 307-338
\bibitem[J1]{J1} D. Jarossay - \emph{A bound on the norm of overconvergent $p$-adic multiple polylogarithms}, arXiv:1503.08756 J. Number Theory 
\bibitem[J2]{J2} D. Jarossay - \emph{Pro-unipotent harmonic actions and a computation of $p$-adic cyclotomic multiple zeta values}, arXiv:1501.04893, submitted
\bibitem[J3]{J3} D. Jarossay - \emph{Pro-unipotent harmonic actions and a dynamical method in the computation of $p$-adic cyclotomic multiple zeta values}, arXiv:1610.09107, to appear in Algebra and Number Theory
\bibitem[J4]{J4} D. Jarossay - \emph{Adjoint cyclotomic multiple zeta values and cyclotomic multiple harmonic values}, arXiv:1412.5099, submitted
\bibitem[J5]{J5} D. Jarossay, \emph{The adjoint quasi-shuffle relation of $p$-adic cyclotomic multiple zeta values recovered via explicit formulas}, arXiv:1601.01158
\bibitem[J6]{J6} D. Jarossay, \emph{Cyclotomic multiple harmonic values regarded as periods}, arXiv:1601.01159
\bibitem[J7]{J7} D. Jarossay, \emph{Around interpolations of $p$-adic multiple zeta values} arXiv:1712.09976
\bibitem[Ka]{Kaneko} M. Kaneko - \emph{Finite multiple zeta values}, RIMS K\^{o}ky\^{u}roku Bessatsu B68 (2017) 175-190
\bibitem[KaZ]{Kaneko Zagier} M. Kaneko, D. Zagier - \emph{Finite multiple zeta values}, preprint
\bibitem[Ko]{Ko} M. Kontsevich, \emph{Holonomic D-modules and positive characteristic}, Japan. J. Math. 4 (2009) 1-25
\bibitem[P]{Panzer} E. Panzer - \emph{The parity theorem for multiple polylogarithms} J. Number Theory, 172 (2017) 93-113
\bibitem[R]{R} C. Reutenauer, \emph{Free Lie algebras}, London Mathematical Society Monographs. New Series, 7, The Clarendon Press Oxford University Press (1993)
\bibitem[Ts]{Tsumura} H. Tsumura, \emph{Combinatorial relations for Euler-Zagier sums} - Acta Arith. - 111 (2004) 1, 27-42
\bibitem[U1]{Unver MZV} S. Ünver - \emph{$p$-adic multi-zeta values} J. Number Theory, 108 (2004) 111-156 
\bibitem[U2]{Unver} S. Ünver, \emph{Drinfel'd-Ihara relations for p-adic multi-zeta values}. J. Number Theory, 133 (2013), 1435-1483 
\bibitem[Yam]{Yamashita} G. Yamashita, \emph{Bounds for the dimension of $p$-adic multiple $L$-values spaces}, Doc. Math., Extra Volume Suslin (2010) 687-723
\bibitem[Y]{Yasuda} S.Yasuda, \emph{Finite real multiple zeta values generate the whole space $Z$}, Int. J. Number Theory 12 (2016) 3, 787-812
\end{thebibliography}
\end{document}